\documentclass[12pt,a4paper,reqno]{amsart}
\usepackage{amssymb}
\usepackage{amscd}
\numberwithin{equation}{section}

\theoremstyle{plain}

\newtheorem{theorem}[subsection]{Theorem}
\newtheorem{proposition}[subsection]{Proposition}
\newtheorem{lemma}[subsection]{Lemma}
\newtheorem{corollary}[subsection]{Corollary}
\newtheorem{conjecture}[subsection]{Conjecture}

\theoremstyle{definition}

\newtheorem{definition}[subsection]{Definition}

\renewcommand{\leq}{\leqslant}
\renewcommand{\geq}{\geqslant}

\newsavebox{\proofbox}
\savebox{\proofbox}{\begin{picture}(7,7)%
  \put(0,0){\framebox(7,7){}}\end{picture}}

\newcommand\E{\mathbb{E}}
\newcommand\Z{\mathbb{Z}}
\newcommand\R{\mathbb{R}}
\newcommand\T{\mathbb{T}}
\newcommand\C{\mathbb{C}}
\newcommand\N{\mathbb{N}}

\newcommand\X{\mathcal{X}}
\newcommand\Y{\mathcal{Y}}
\newcommand\U{\mathcal{U}}

\newcommand\B{\mathcal{B}}

\newcommand\F{\mathbb{F}}

\newcommand\eps{\varepsilon}

\begin{document}

\title{Multiple recurrence and convergence results associated to $\F_{p}^{\omega}$-actions}


\author{Vitaly Bergelson}
\address{Department of Mathematics, The Ohio-State University, Columbus OH 43210}
\email{vitaly@@math.ohio-state.edu}

\author{Terence Tao}
\address{Department of Mathematics, UCLA, Los Angeles CA 90095-1555}
\email{tao@@math.ucla.edu}

\author{Tamar Ziegler}
\address{Department of Mathematics, Technion, Haifa Israel 32000}
\email{tamarzr@@tx.technion.ac.il}

\begin{abstract}  Using an ergodic inverse theorem obtained in our previous paper, we obtain limit formulae for multiple ergodic averages associated with the action of $\F_{p}^{\omega}=\oplus \F_p$.  From this we deduce multiple Khintchine-type recurrence results analogous to those for $\Z$-systems obtained by Bergelson, Host, and Kra, and also present some new counterexamples in this setting.
\end{abstract}

\maketitle

\section{Introduction}

The celebrated Szemer\'edi's theorem \cite{szemeredi} stating that any set of positive density in $\mathbb Z$ contains arbitrary long progressions, has a natural analogue for ``large'' sets in the group $\F_{p}^{\omega}=\oplus \F_p$, the direct sum of countably many copies of a finite field of prime order $p$. While the content of Szemer\'edi's theorem can be succinctly expressed by the maxim ``large sets in $\mathbb Z$ are AP-rich'' (where AP stands for Arithmetic Progression), the $\F_{p}^{\omega}$ analogue states that any ``large'' set in $\F_{p}^{\omega}$ is AS-rich, that is, it contains arbitrarily large Affine Subspaces. This analogy extends to the similitude between various proofs of these two theorems and is especially interesting when one studies the $\F_{p}^{\omega}$ analogues of various aspects of the ergodic approach to Szemer\'edi's theorem introduced by Furstenberg in \cite{furstenberg}. 

Given an invertible probability measure preserving system $(X,\X,\mu,T)$, and a set $A \in \X$ with $\mu(A)>0$ and
an integer $k \in \N$, let $\phi(n)=\mu(A \cap T^nA \cap \ldots \cap T^{kn}A)$. The sequence $\phi(n)$ can be viewed as a generalized positive definite sequence. The analysis of the properties and the asymptotic behavior of $\phi(n)$ leads to the proof and enhancements of Szemer\'edi's theorem. Not surprisingly, the study of the $\F_{p}^{\omega}$ analogue of $\phi(n)$ leads to the better understanding and enhancement of the $\F_{p}^{\omega}$ analogue of Szemer\'edi's theorem. It also throws new light on the various related facts belonging to the realm of ergodic theory.

In this paper we will describe (in Theorem \ref{char-thm} below) the characteristic factor for certain multiple ergodic averages on measure-preserving systems, in the case where the underlying group $G$ is an infinite-dimensional vector space $\F_p^\omega$ over a finite field; this is the analogue of the well-known description in  \cite{host-kra}, \cite{ziegler} of characteristic factors for multiple ergodic averages of $\Z$-actions.  Using this description, we can obtain explicit formulae for the limit of such multiple ergodic averages.  As an application of these formulae, we can establish multiple recurrence theorems of Khintchine type in some cases, and establish counterexamples to such theorems in other cases.

The detailed statements of the main results of our paper will be formulated at the end of the introduction.

\subsection{Convergence of multiple ergodic averages and limit formulae} 

Before we can properly state our main results, we need to set up a certain amount of notation regarding measure-preserving $G$-systems and their characteristic factors.

Let $G = (G,+)$ be a countable abelian group, and let $(X,\X,\mu)$ be a probability space, which we will always assume to be \emph{separable}\footnote{This assumption is not used explicitly in this paper, but is used in the paper \cite{bergelson-tao-ziegler} whose results we rely on, in order to perform certain measurable selections (see Appendix C of that paper) as well as disintegrations of measures.} in the sense that the $\sigma$-algebra $\X$ is countably generated modulo $\mu$-null sets; in most applications one can reduce to this situation without difficulty, so this will not be a serious restriction in practice.  An \emph{invertible measure-preserving transformation} on $X$ is an invertible map $T: X \to X$ with $T$ and $T^{-1}$ both measurable, such that $\mu(T_g(E))=\mu(E)$ for all $E \in \X$.  A \emph{measure-preserving $G$-action} on $X$ is a family $(T_g)_{g \in G}$ of invertible measure-preserving transformations $T_g: X \to X$, such that $T_g T_h = T_{g+h}$ and $T_0 = \operatorname{id}$ $\mu$-almost everywhere for all $g,h \in G$.  We refer to the quadruplet $\mathrm{X} = (X, \X, \mu, (T_g)_{g \in G})$ as a \emph{measure-preserving $G$-system}, or \emph{$G$-system} for short.  We abbreviate the (complex-valued) Lebesgue spaces $L^p(X,\X,\mu)$ for $1 \leq p \leq \infty$ as $L^p(\mathrm{X})$.  We adopt the usual convention of identifying two functions in $L^p(\mathrm{X})$ if they agree $\mu$-almost everywhere; in particular, this makes $L^2(\mathrm{X})$ a separable Hilbert space.

A \emph{F{\o}lner sequence} in $G$ is a sequence $(\Phi_n)_{n=1}^\infty$ of finite non-empty subsets of $G$ such that $$ \lim_{n \to \infty} \frac{|(g + \Phi_n) \Delta \Phi_n|}{|\Phi_n|} = 0$$
for all $g \in G$, where $|H|$ denotes the cardinality of a finite set $H$.  Note that we do not require the $\Phi_n$ to be nested, or to exhaust all of $G$.  Note that the class of F{\o}lner sequences is translation-invariant in the sense that if $(\Phi_n)_{n=1}^\infty$ is a F{\o}lner sequence, then $(g_n+\Phi_n)_{n=1}^\infty$ is also a F{\o}lner sequence for any $g_1, g_2, \ldots \in G$.  It is a classical fact that every countable abelian group is amenable \cite{neumann} and hence has at least one F{\o}lner sequence \cite{folner}; for instance, if $G=\Z$, one can take $\Phi_n := \{1,\ldots,n\}$.

The classical mean ergodic theorem\footnote{This theorem is usually stated in textbooks for $\Z$-systems, but the proof extends without difficulty to actions by other amenable groups. See for example \cite[Theorem 6.4.15]{bergelson}.} asserts, among other things, that if $G = (G,+)$ is a countable abelian group with F{\o}lner sequence $(\Phi_n)_{n=1}^\infty$, $\mathrm{X}=(X,\X,\mu,(T_g)_{g\in G})$ is a $G$-system and $f \in L^2(\mathrm{X})$, then the limit
\begin{equation}\label{tgf}
 \lim_{n \to \infty} \E_{g \in \Phi_n} T_g f
\end{equation}
converges strongly in $L^2(\mathrm{X})$ norm, where we use the averaging notation
$$
\E_{h \in H} := \frac{1}{|H|} \sum_{h \in H}
$$
for any non-empty finite set $H$, and also write $T_g f$ for $f \circ T_g$.  Since strong convergence in $L^2(\mathrm{X})$ implies weak convergence, we obtain as a corollary that the limit
\begin{equation}\label{tgf2}
 \lim_{n \to \infty} \E_{g \in \Phi_n} \int_X f_0 T_g f_1\ d\mu
\end{equation}
exists for all $f_0,f_1 \in L^2(\mathrm{X})$.

The mean ergodic theorem not only gives existence of these limits, but provides a formula for the value of these limits.  To describe this formula we need some more notation. Define a \emph{factor} $(Y,\Y,\nu,(S_g)_{g \in G},\pi) = (\mathrm{Y},\pi)$ of a $G$-system $(X,\X,\mu,(T_g)_{g \in G})$ to be another $G$-system $\mathrm{Y} = (Y,\Y,\nu,(S_g)_{g \in G})$, together with a measurable map $\pi: X \to Y$ which respects the measure in the sense that $\mu(\pi^{-1}(E)) = \nu(E)$ for all $E \in \Y$ (or equivalently, $\pi_* \mu= \nu$), and also respects the $G$-action in the sense that $S_g \circ \pi = \pi \circ T_g$ $\mu$-a.e. for all $g \in G$.  For instance, if $\B$ is a sub-$\sigma$-algebra of $\X$ which is invariant with respect to the $G$-action $(T_g)_{g \in G}$, then 
$(X, \B, \mu\downharpoonright_\B, (T_g)_{g \in G}, \operatorname{id})$ is a factor of $(X,\X,\mu, (T_g)_{g \in G})$, where $\mu\downharpoonright_\B$ is the restriction of $\X$ to $\B$.  By abuse of notation, we will thus refer to an invariant sub-$\sigma$-algebra $\B$ as a factor of $(X,\X,\mu, (T_g)_{g \in G})$.  We call two factors $(Y,\Y,\nu,(S_g)_{g \in G},\pi)$, $(Y',\Y',\nu',(S'_g)_{g \in G},\pi')$ \emph{equivalent} if the sub-$\sigma$-algebras $\{ \pi^{-1}(E): E \in \Y \}$, $\{ (\pi')^{-1}(E): E \in \Y' \}$ of $\X$ that they generate agree modulo null sets.  It is clear that every factor is equivalent to a unique invariant (modulo null sets) sub-$\sigma$-algebra of $\X$, so one may think of factors as invariant sub-$\sigma$-algebras if it is convenient to do so.

Given a factor $(Y,\Y,\nu,(S_g)_{g \in G},\pi) = (\mathrm{Y},\pi)$, we have a pullback map $\pi^*: L^2(\mathrm{Y}) \to L^2(\mathrm{X})$ defined by $\pi^* f := f \circ \pi$.  We define the pushforward map $\pi_*: L^2(\mathrm{X}) \to L^2(\mathrm{Y})$ to be the adjoint of this map.  In the case when the factor arises from an invariant sub-$\sigma$-algebra $\B$ of $\X$, the pushforward $\pi_* f$ is the same as the conditional expectation $\E(f|\B)$ of $f$ with respect to $\B$.

Given a $G$-system $\mathrm{X} = (X,\X,\mu,(T_g)_{g \in G})$, we define the \emph{invariant factor} $(\mathrm{Z}_0,\pi_0) = (\mathrm{Z}_0(\mathrm{X}),\pi_0) = (Z_0,\mathcal{Z}_0,\mu_0,(T_{g})_{g \in G},\pi_0)$ of $\mathrm{X}$ to be (up to equivalence\footnote{One can of course define $(\mathrm{Z}_0,\pi_0)$ canonically to actually \emph{be} the factor $(X, \X^T, \mu\downharpoonright_{\X^T}, (T_g)_{g \in G}, \operatorname{id})$, but it can be convenient to allow $\mathrm{Z}_0$ to only be defined up to equivalence, in order to take advantage of other models of the invariant factor which may be more convenient to compute with.}) the factor associated to the invariant $\sigma$-algebra $\X^T := \{ E \in \X: T_g E = E \hbox{ for all } g \in G \}$.    This factor is a \emph{characteristic factor} for the averages \eqref{tgf}, \eqref{tgf2}, in the sense that the limit in \eqref{tgf} converges strongly in $L^2(\mathrm{X})$ to zero whenever $(\pi_0)_* f$ vanishes, and similarly the limit in \eqref{tgf2} converges to zero when either of $(\pi_0)_* f_0$ or $(\pi_0)_* f_1$ vanishes (see \cite{fw}).  As a consequence, to compute the limits in \eqref{tgf}, one may freely replace $f$ by $(\pi_0)_* f$ (and descend from $\mathrm{X}$ to the factor $\mathrm{Z}_0$), and similarly for \eqref{tgf2}.  On the characteristic factor $\mathrm{Z}_0$, the action of $G$ is essentially trivial, and as a conclusion one obtains the well-known limit formulae
$$  \lim_{n \to \infty} \E_{g \in \Phi_n} T_g f = (\pi_0)^* (\pi_0)_* f$$
and
$$ \lim_{n \to \infty} \E_{g \in \Phi_n} \int_X f_0 T_g f_1\ d\mu = \int_{Z_0} \left((\pi_0)_* f_0\right) \left((\pi_0)_* f_1\right)\ d(\pi_0)_*\mu.$$

The situation is particularly simple when the $G$-system $\mathrm{X}$ is \emph{ergodic}, which means that the invariant $\sigma$-algebra $\X^T$ consists only of sets of full measure or empty measure, or equivalently that the invariant factor $\mathrm{Z}_0$ is  a point.  In this case, $(\pi_0)_* f = \int_X f\ d\mu$, and so
$$  \lim_{n \to \infty} \E_{g \in \Phi_n} T_g f = \int_X f\ d\mu$$
and
$$ \lim_{n \to \infty} \E_{g \in \Phi_n} \int_X f_0 T_g f_1\ d\mu = \left(\int_X f_0\ d\mu\right) \left(\int_X f_1\ d\mu\right).$$

This concludes our discussion of the classical ergodic averages. We now consider the more general \emph{multiple ergodic averages}
\begin{equation}\label{tgf-mult}
 \lim_{n \to \infty} \E_{g \in \Phi_n} (T_{c_1g} f_1) (T_{c_2g} f_2) \ldots (T_{c_kg} f_k)
\end{equation}
and
\begin{equation}\label{tgf2-mult}
 \lim_{n \to \infty} \E_{g \in \Phi_n} \int_X (T_{c_0 g} f_0) \ldots (T_{c_k g} f_k) \ d\mu
\end{equation}
associated to a $G$-system $\mathrm{X} = (X, \X, \mu, (T_g)_{g \in G})$, where $k \geq 1$ and $c_0,\ldots,c_k$ are integers, and (to avoid absolute integrability issues) $f_0,\ldots,f_k$ will be assumed now to lie in $L^\infty(\mathrm{X})$ rather than $L^2(\mathrm{X})$.  Note that in \eqref{tgf2-mult} we may collect terms if necessary and reduce to the case when the $c_0,\ldots,c_k$ are distinct.  Similarly, in \eqref{tgf-mult} we may reduce to the case when the $c_1,\ldots,c_k$ are distinct and non-zero (since zero coefficients can simply be factored out).  The reader can keep the model case $c_i=i$ in mind for this discussion, though for technical reasons it is convenient to consider more general coefficients $c_i$ as well.

The convergence and recurrence properties of these averages have been extensively studied in the literature, particularly in the model case $G=\Z$.  For instance, the celebrated Furstenberg multiple recurrence theorem \cite{furstenberg} asserts the lower bound
$$
 \liminf_{n \to \infty} \E_{g \in \Phi_n} \int_X f (T_g f) \ldots (T_{kg} f) \ d\mu > c>0
$$
in the $G=\Z$ case, whenever $k \geq 1$ and $f \in L^\infty(\mathrm{X})$ is non-negative and not identically zero, and where $c$ does not depend on the choice of $(\Phi_n)_{n=1}^{\infty}$.  The same result holds for all countable abelian groups $G$ \cite{furstenberg-katznelson}.  On the other hand, the original proofs of the multiple recurrence theorem did not actually establish the existence of the limit in \eqref{tgf-mult} or \eqref{tgf2-mult} for general $k > 1$.  In the case of $\Z$-actions, this was first achieved for $k=2$ in \cite{furstenberg}, for $k=3$ in \cite{zhang} (building upon a sequence of partial results in \cite{cl84,cl87,cl88,fw}). The case $k=4$ was established in \cite{hk} (see also \cite{hk4}) and independently in \cite{ziegler-t}. The methods in \cite{hk}, \cite{ziegler-t} were generalized to cover all $k \ge 1$ first in \cite{host-kra} and then in \cite{ziegler}.  After the general convergence of \eqref{tgf-mult}, \eqref{tgf2-mult} for $G=\Z$ was established, a number of additional proofs of this result (as well as generalisations thereof) have appeared in the literature  \cite{tao}, \cite{austin}, \cite{towsner}, \cite{host}, \cite{walsh}.  The argument in \cite{walsh} is in fact extremely general, and extends to averages over arbitrary countable abelian groups $G$, with the shifts $g, \ldots, (k-1) g$ replaced by polynomial functions of $g$ (see also \cite{zorin}).

Now we turn to the question of understanding the nature of the limit in \eqref{tgf-mult} or \eqref{tgf2-mult} for higher values of $k$ than $k=1$.  For simplicity, we will focus on the case of ergodic $G$-systems $\mathrm{X} = (X,\X,\mu,(T_g)_{g \in G})$; the results discussed here can then be extended to the non-ergodic case by ergodic decomposition (see e.g. \cite{varadarajan}).

The case $k=2$ can be analysed by spectral theory.  Define an \emph{eigenfunction} of an ergodic $G$-system $\mathrm{X}$ to be
a non-zero function $f \in L^2(\mathrm{X})$ such that for each $g \in G$ one has $T_g f = \lambda_g f$ for some complex number $\lambda_g$.  
Define the \emph{Kronecker factor} $(\mathrm{Z}_1,\pi_1) = (\mathrm{Z}_1(\mathrm{X}),\pi_1) = (Z_1, \Z_1, \mu_1, (T_{1,g})_{g \in G}, \pi_1)$ of $\mathrm{X}$ to be the factor (up to equivalence) associated to the sub-$\sigma$-algebra of $\X$ generated by the eigenfunctions of $\mathrm{X}$.  The Kronecker factor is (up to equivalence) is given by an abelian group rotation
$$ \mathrm{Z}_1 = (U, \U, m_U, (S_g)_{g \in G})$$
where $U = (U,+)$ is a compact abelian group with Borel $\sigma$-algebra $\U$ and Haar probability measure $m_U$, and each $S_g: U \to U$ is a group translation $S_g(x) := x + \alpha_g$, where $g \mapsto \alpha_g$ is a homomorphism from $G$ to $U$ (see \cite{zimmer} for a general form of this theorem).  Furthermore, this factor is ergodic (which is equivalent to the image of the homomorphism $g \mapsto \alpha_g$ being dense in $U$).  It is known (see e.g. \cite{bergelson-tao-ziegler}) that the Kronecker factor $\mathrm{Z}_1$ is characteristic for the $k=2$ averages \eqref{tgf-mult}, \eqref{tgf2-mult}, in the sense that the former average converges to zero in $L^2(\mathrm{X})$ norm when at least one of $(\pi_1)_* f_1, (\pi_1)_* f_2$ vanishes, and the latter average converges to zero when at least one of $(\pi_1)_* f_0, (\pi_1)_* f_1, (\pi_1)_* f_2$ vanishes.  From this, one can effectively replace each function $f_i$ by its pushforward $(\pi_1)_* f_i$ in the limits \eqref{tgf-mult}, \eqref{tgf2-mult} (and replacing $\mathrm{X}$ with $\mathrm{Z}_1$).  These limits can then be evaluated by harmonic analysis on $U$, resulting in the limit formula
\begin{equation}\label{limit-2}
\lim_{n \to \infty} \E_{g \in \Phi_n} \int_X f_0 (T_g f_1) (T_{2g} f_2) \ d\mu
= \int_U \int_U (\pi_1)_* f_0(h) (\pi_1)_* f_1(h+t) (\pi_1)_* f_2(h+2t)\ d\mu_U(h) d\mu_U(t)
\end{equation}
for \eqref{tgf2-mult} (in the model case $c_i=i$), and hence (by duality, and existence of the limit) a similar formula for \eqref{tgf-mult}; similarly for other choices of coefficients $c_i$.  We can rewrite this formula as
\begin{equation}\label{limit-2a}
\begin{aligned}
&\lim_{n \to \infty} \E_{g \in \Phi_n} \int_X f_0 (T_g f_1) (T_{2g} f_2) \ d\mu \\
&= \int_{HP_{0,1,2}(U)} (\pi_1)_* f_0(h_0) (\pi_1)_* f_1(h_1) (\pi_1)_* f_2(h_2)\ dm_{HP_{0,1,2}(U)}(h_0,h_1,h_2)
\end{aligned}\end{equation}
where $HP_{0,1,2}(U) \subset U^3$ is the closed subgroup
$$ HP_{0,1,2}(U) := \{ (h,h+t,h+2t): h,t \in U \}$$
of $U^3$, and $m_{HP_{0,1,2}(U)}$ is the Haar probability measure on $HP_{0,1,2}(U)$.  (The reason for the notation $HP_{0,1,2}$ will be made clearer later.)

In the case of $\Z$-actions, the limit of \eqref{tgf-mult}, \eqref{tgf2-mult} for higher values of $k$ is also understood; see \cite{ziegler-0}, \cite{ziegler}, \cite{bhk}.  For each value of $k$, a characteristic factor $\mathrm{Z}_k$ associated to the averages \eqref{tgf-mult}, \eqref{tgf2-mult} which (up to equivalence) is an inverse limit of nilsystems of step at most $k-1$ was constructed in \cite{host-kra} (see also \cite{ziegler}).  Projecting onto each such nilsystem and using the equidistribution theory on such nilsystems (see \cite{leibman}, \cite{ziegler-0}) a limit formula generalising \eqref{limit-2}, \eqref{limit-2a} (but for $\Z$-actions) was established; see \cite{ziegler}.   A closely related analysis was also performed in \cite{bhk}, which among other things led to the following Khintchine-type recurrence result: if $\mathrm{X} = (X, \X, \mu, (T_g)_{g \in \Z})$ is an ergodic $\Z$-system and $A \in \X$ has positive measure, then for every $\eps > 0$ and all $k=1,2,3$, the sets\footnote{The negative signs here are artifacts of our sign conventions, and can be easily removed if desired.}
$$ \{ n \in \Z: \mu( A \cap T_{-n} A \cap \ldots \cap T_{-kn} A ) \geq \mu(A)^{k+1} - \eps \}$$
are syndetic.  Surprisingly, this type of result fails for $k>3$; see \cite{bhk} for details.  

The arguments in \cite{bhk} also give a structural result for the correlation sequences\footnote{Strictly speaking, the results in \cite{bhk} are only claimed in the case $f_0=\ldots=f_{k}$ and $c_i=i$, but it is not difficult to see that the argument in fact applies in general.}
\begin{equation}\label{igdef}
 I_{c_0,\ldots,c_k; f_0,\ldots,f_k}(g) := \int_X (T_{c_0 g} f_0) (T_{c_1g} f_1) \ldots (T_{c_kg} f_k)\ d\mu
\end{equation}
for $k \geq 1$, $f_0,\ldots,f_k \in L^\infty(\mathrm{X})$ and $g \in \Z$ and distinct integers $c_1,\ldots,c_k$.  To state these results, recall that a \emph{$k-1$-step nilsequence} a uniform limit of sequences of the form $n \mapsto F(\theta^n \Gamma)$ for a $k-1$-step nilmanifold $N/\Gamma$, a group element $\theta \in N$, and a continuous function $F: N/\Gamma \to \C$; recall also that a bounded sequence $\sigma: G \to \C$ in a countable abelian group $G$ is said to converge to zero in uniform density if one has
$$ \lim_{n \to \infty} \sup_{h \in G} \E_{g \in h + \Phi_n} |\sigma(g)| = 0$$
for any F{\o}lner sequence $(\Phi_n)_{n=1}^\infty$.  Then it was shown in \cite{bhk} that the sequence $I_{c_0,\ldots,c_k; f_0,\ldots,f_k}$ can be decomposed as the sum of a $k-1$-step nilsequence and an error sequence $n \mapsto \sigma(n)$ which converges to zero in uniform density.  

\subsection{New results}

Having reviewed the preceding results, we now proceed to the description of new results in this paper, in which we focus on a family of countable abelian groups $G$ at the opposite end of the spectrum to the integers $\Z$, namely the infinite-dimensional vector space $G := \F_p^\omega=\oplus \F_p$ over a finite field $\F_p$ of prime order $p$, with a countable basis $e_1,e_2,\ldots$.  This can be viewed as the direct limit\footnote{Note that this limit is distinct from the \emph{inverse limit} $\prod_{i=1}^\infty \F_p$ of the $\F_p^n$; for instance, $\F_p^\omega$ is a countable vector space, whereas $\prod_{i=1}^{\infty} \F_p$ is an uncountable (but compact) group.}  of the finite-dimensional subspaces $\F_p^n$, defined as the span of $e_1,\ldots,e_n$, and a $G$-system can be viewed as a probability space with an infinite sequence $T_{e_n}: X \to X$ of commuting measure-preserving transformations, each of period $p$ in the sense that $T_{e_n}^p=\operatorname{id}$.  Observe that we can view these subspaces $\F_p^n$ as a F{\o}lner sequence for $\F_p^\omega$, but this is of course not the only such sequence (for instance, one can take the affine spaces $g_n + \F_p^n$, where $g_1, g_2, \ldots$ is an arbitrary sequence in $\F_p^\omega$).  One can then ask for a formula for the limits in \eqref{tgf-mult}, \eqref{tgf2-mult}, as well as a structure theorem for the correlation sequences \eqref{igdef} (now defined for $g \in G$ rather than $g \in \Z$).  

To state the results, we need to introduce a variant of the concept of nilsystem that is suitable for $\F_p^\omega$-actions, which we refer to as a \emph{Weyl system}.  To define such systems, we first need the notion of a \emph{polynomial function} on a $G$-system.

\begin{definition}[Polynomials] Let $G = (G,+)$ be a countable abelian group, let $U = (U,+)$ be an abelian group, and let $\mathrm{X} = (X,\X,\mu,(T_g)_{g \in G})$ be a measure preserving system.  For any measurable function $\rho: X \to U$ and $g \in G$, let $\Delta_g \rho: X \to U$ denote the function
$$ \Delta_g \rho(x) := \rho(T_g x) - \rho(x),$$
thus $\Delta_g$ can be viewed as a difference operator on the measurable functions from $X$ to $U$. 
If $k \geq 1$ is a natural number, we say that $\rho$ is a \emph{polynomial of degree $<k$} if $\Delta_{g_1} \ldots \Delta_{g_k} \rho(x) =0$ $\mu$-almost everywhere for any $g_1, \ldots, g_k \in G$.  We also adopt the convention that the zero function is the only polynomial of degree $<k$ if $k \leq 0$.

In a similar vein, a sequence $g: \Z \to U$ is said to be a polynomial of degree $<k$ if $\Delta_{h_1} \ldots \Delta_{h_k} g(n) = 0$ for all $h_1,\ldots,h_k,n \in \Z$, where $\Delta_h g(n) := g(n+h) - g(n)$, with the same convention as before if $k \leq 0$.
\end{definition}

Note that a measurable function $\rho: X \to \R/\Z$ is a polynomial of degree $<2$ if and only if the function $e^{2\pi i \rho}$ is an eigenfunction of the system $\mathrm{X}$.  Thus we see that the polynomials of degree $<2$ are closely related to the Kronecker factor, which in turn controls the $k=2$ averages \eqref{tgf-mult}, \eqref{tgf2-mult}.  More generally, we will see (in the case $G = \F_p^\omega$) that the polynomials of degree $<k$ control the averages  \eqref{tgf-mult}, \eqref{tgf2-mult}.  One can define polynomial maps between more general groups (not necessarily abelian); see \cite{leib-alg}.  However, we will not need this more general concept of a polynomial map here.

For future reference we observe (by an easy induction using Pascal's triangle) that a sequence $g: \Z \to U$ is a polynomial of degree $<k$ if and only if it has a discrete Taylor expansion of the form
$$ g(n) = \sum_{0 \leq j < k} \binom{n}{j} a_j$$
for some coefficients $a_j \in U$, where $\binom{n}{j} := \frac{n(n-1)\ldots(n-j+1)}{j!}$.  We remark that the top coefficient $a_{k-1}$ of $g(n)$ can also be computed as $a_{k-1} = \Delta_1^{k-1} g(n)$ for any $n$.

Next, we recall the notion of a \emph{cocycle extension}.

\begin{definition}[Cocycle extension]  Let $G = (G,+)$ be a countable abelian group, let $U = (U,+)$ be a compact abelian group, and let $\mathrm{X} = (X,\X,\mu,(T_g)_{g \in G})$ be a measure preserving system.  A \emph{$(G,\mathrm{X},U)$-cocycle} is a measurable function $\rho: G \times X \to U$ that obeys the \emph{cocycle equation}
\begin{equation}\label{cocycle}
 \rho(g+g', x) = \rho(g, T_{g'} x) + \rho(g', x)
\end{equation}
for all $g,g' \in G$ and $\mu$-almost all $x \in X$.  Given such a cocycle, we define the \emph{extension} $\mathrm{X} \times_\rho U$ of $\mathrm{X}$ by the cocycle $\rho$ to be the $G$-system given by the product probability space
$$ (X \times U, \X \times \U, \mu \times m_U)$$
(where $\U$ is the Borel $\sigma$-algebra on $U$, and $m_U$ the Haar probability measure on $U$), and the action $(\tilde T_g)_{g \in G}$ on $X \times U$ given by the formula
$$ \tilde T_g ( x, u ) := (T_g x, u + \rho(g,x)).$$
Note that the cocycle equation \eqref{cocycle} ensures that $\mathrm{X} \times_\rho U$ is indeed a $G$-system.  If $k$ is a positive integer, we say that the cocycle $\rho$ is a \emph{polynomial cocycle of degree $<k$} if, for each $g \in G$, the function $x \mapsto \rho(g,x)$ is a polynomial of degree $<k$.
\end{definition}

\begin{definition}[Weyl system]  Let $k \geq 0$ be an integer, and let $G = (G,+)$ be a countable abelian group.  We define a \emph{$k$-step Weyl $G$-system} recursively as follows:
\begin{itemize}
\item A $0$-step Weyl $G$-system is a point.
\item If $k \geq 1$, a $k$-step Weyl $G$-system is any system of the form $\mathrm{X} \times_{\rho_k} U_k$, where $\mathrm{X}$ is a Weyl $G$-system of order $k-1$, $U_k$ is a compact abelian group\footnote{In this paper, ``compact group'' is understood to be short for ``compact metrizable group''.}, and $\rho_k$ is a polynomial $(G, \mathrm{X}, U_k)$-cocycle of degree $<k$.
\end{itemize}
We define the notion of a \emph{continuous $k$-step Weyl system} similarly to a $k$-step Weyl system, except now that all the cocycles involved are also required to be continuous.  (Note that a Weyl system is a Cartesian product of compact spaces and is thus also compact.)
\end{definition}

Informally, a Weyl $G$-system of order $k$ takes the form
$$ U_1 \times_{\rho_2} U_2 \times_{\rho_3} \ldots \times_{\rho_k} U_k$$
for some compact abelian groups $U_1,\ldots,U_k$ (which we refer to as the \emph{structure groups} of the system) and polynomial cocycles $\rho_1,\ldots,\rho_k$ (the cocycle $\rho_1$ is essentially a homomorphism from $G$ to $U_1$ and is not explicitly shown in the above notation).  In the $k=1$ case, a Weyl $G$-system is simply a group rotation $T_g: u_0 \mapsto u_0 + \rho_1(g)$ on $U_1$.  

\begin{remark} In \cite{bergelson-tao-ziegler} we define the notion of an {\em Abramov  $\F_p^{\omega}$-system} $\text{Abr}_{<k}(\mathrm{X})$.  This is is a system where $P_{<k}(X)$ - the polynomials of degree $<k$ - span $L^2(X)$. We show that in the case where $k \le \text{char}(\F)$ an Abramov system can be given the structure of a Weyl system. \end{remark}

\begin{example}  Let $\mathrm{X}_1$ be the product space $\prod \F_p$ of sequences $(x_n)_{n=1}^\infty$ with $x_n \in \F_p$, with the product topology and the Haar probability measure.  It becomes a $1$-step Weyl $G$-system with $G := \F_p^\omega$ by using the shifts
$$ T_g ( x_n )_{n=1}^\infty := (x_n+g_n)_{n=1}^\infty$$
when $g = \sum_{n=1}^\infty g_n e_n$ with $g_n \in \F_p$ (and with all but finitely many of the $g_n$ vanishing).

Formally, we have a quadratic polynomial $Q: \prod \F_p \to \F_p$ defined by
$$ Q( (x_n)_{n=1}^\infty ) = \sum_{n=1}^\infty x_n x_{n+1},$$
however this polynomial is not actually well defined because the sum here can contain infinitely many non-zero terms.  However, if we compute a formal derivative $\Delta_g Q$ of this polynomial for some $g = \sum_{n=1}^\infty g_n e_n \in G$, we obtain
$$ \Delta_g Q = \sum_{n=1}^\infty g_n x_{n+1} + x_n g_{n+1} + g_n g_{n+1},$$
and this is a well-defined linear polynomial on $\prod \F_p$ because only finitely many of the $g_n$ are non-zero.  If we set $\rho_2(g,x) := \Delta_g Q(x)$, we see that $\rho_2(g,x)$ is a polynomial $(G,\mathrm{X}_1,\F_p)$-cocycle of degree $<2$.  If we then take the cocycle extension $\mathrm{X}_2:= \mathrm{X}_1 \times_2 \F_p$, then $\mathrm{X}_2$ is a $2$-step Weyl system with structure groups $\prod \F_p$ and $\F_p$, with shift given by
$$ T_g( (x_n)_{n=1}^\infty, t ) = ( (x_n+g_n)_{n=1}^\infty, t + \sum_{n=1}^\infty g_n x_{n+1} + x_n g_{n+1} + g_n g_{n+1}).$$
This system  can be viewed as a $G$-system analogue to a $2$-step nilsystem arising from the Heisenberg group.
\end{example}

The first main result, which is a corollary of our previous work in \cite{bergelson-tao-ziegler}, establishes the existence of a Weyl system as a characteristic factor for the averages \eqref{tgf-mult}, \eqref{tgf2-mult}:

\begin{theorem}[Characteristic factor]\label{char-thm}
Let $p$ be a prime, and let $1 \leq k < p$ be an integer.  Let $G := \F_p^\omega$, and let $\mathrm{X} = (X,\X,\mu,(T_g)_{g \in G})$ be an ergodic $G$-system.  Then for each $0 \leq k < p$, there exists a factor $(\mathrm{Z}_k,\pi_k) = (\mathrm{Z}_k(\mathrm{X}),\pi_k)$ of $\mathrm{X}$, with $\mathrm{Z}_k$ an ergodic continuous $k$-step Weyl system, with the following properties:
\begin{itemize}
\item[(i)]  (Recursive description)  One has $\mathrm{Z}_k = \mathrm{Z}_{k-1} \times_{\rho_k} U_k$ for some compact abelian group $U_k$ and some polynomial $(G, \mathrm{Z}_{k-1}, U_k)$-cocycle $\rho_k$ of degree $<k$.  Furthermore, $U_k$ is a $p$-torsion group (thus $pu_k=0$ for all $u_k \in U_k$, or equivalently $U_k$ is a vector space over $\F_p$).
\item[(ii)]  (Connection with polynomials) The sub-$\sigma$-algebra of $\X$ generated by $\mathrm{Z}_{k}$ is generated by the polynomials $\phi: X \to \R/\Z$ of degree $<k+1$.  (Thus, for instance, $\mathrm{Z}_1$ is the Kronecker factor.)
\item[(iii)]  ($\mathrm{Z}_{k-1}$ characteristic for \eqref{tgf-mult}) For distinct non-zero $c_1,\ldots,c_k \in \F_p \backslash \{0\}$, the averages \eqref{tgf-mult} converge strongly in $L^2(\mathrm{X})$ to zero for any F{\o}lner sequence $(\Phi_n)_{n=1}^\infty$ of $G$, whenever $f_1,\ldots,f_k \in L^\infty(\mathrm{X})$ is such that $(\pi_{k-1})_* f_i = 0$ for at least one $i=1,\ldots,k$;
\item[(iv)] ($\mathrm{Z}_{k-1}$ characteristic for \eqref{tgf2-mult}) For distinct $c_0,\ldots,c_k\in \F_p$, the average \eqref{tgf2-mult} converges to zero for any F{\o}lner sequence $(\Phi_n)_{n=1}^\infty$ of $G$, whenever $f_0,\ldots,f_k \in L^\infty(\mathrm{X})$ is such that $(\pi_{k-1})_* f_i = 0$ for at least one $i=0,\ldots,k$;
\item[(v)] ($\mathrm{Z}_{k}$ characteristic for \eqref{igdef}) For distinct $c_0,\ldots,c_k \in \F_p$, whenever $f_0,\ldots,f_{k} \in L^\infty(\mathrm{X})$ is such that $(\pi_{k})_* f_i = 0$ for some $i=0,\ldots,k-1$, the sequence $I_{c_0,\ldots,c_k;f_0,\ldots,f_{k}}: G \to \C$ converges to zero in uniform density.
\end{itemize}
\end{theorem}

We establish this theorem in Section \ref{char-sec}.  Now we discuss some consequences of Theorem \ref{char-thm}.  We begin with a limit formula for the average \eqref{tgf2-mult}.  We will need the following construction:

\begin{definition}[Hall-Petresco groups]  Let $p$ be a prime, and let $U_1,\ldots,U_m$ be compact $p$-torsion groups for some $0 \leq m < p$.  Let $1 \leq k < p$, and let $c_0,\ldots,c_k \in \F_p$ be distinct.  Define the \emph{Hall-Petresco group} $HP_{c_0,\ldots,c_k}(U_1,\ldots,U_m)$ to be the closed subgroup of $(U_1 \times \ldots\times U_m)^{k+1}$ consisting of tuples of the form $(P(c_i))_{i=0}^k$, where $P = (P_1,\ldots,P_m): \Z \to U_1 \times \ldots \times U_m$ and for each $1 \leq j \leq m$, $P_j: \Z \to U_j$ is a polynomial of degree $<j+1$.

If $\mathrm{X} = U_1 \times_{\rho_2} U_2 \times_{\rho_3} \ldots \times_{\rho_m} U_m$ is an ergodic $m$-step Weyl system, we abbreviate
$HP_{c_0,\ldots,c_k}(U_1,\ldots,U_m)$ as $HP_{c_0,\ldots,c_k}(\mathrm{X})$.
\end{definition}

Thus, for instance
\begin{align*}
&HP_{0,1}(U_1, U_2, U_3 )\\
&\quad = \left\{ \left((a_1,a_2,a_3), (a_1+b_1, a_2+b_2, a_3+b_3)\right): a_1,b_1 \in U_1; a_2,b_2 \in U_2; a_3,b_3 \in U_3 \right\} \\
&\quad = (U_1 \times U_2 \times U_3)^2,
\end{align*}
and (for $p>2$)
\begin{align*}
&HP_{0,1,2}(U_1, U_2, U_3 )\\
&\quad= \biggl\{ \left((a_1,a_2,a_3), (a_1+b_1,a_2+b_2,a_3+b_3), (a_1+2b_1, a_2+2b_2+c_2, a_3+2b_3+c_3)\right): \\
&\quad\quad a_1,b_1 \in U_1; a_2,b_2,c_2 \in U_2; a_3,b_3,c_3 \in U_3 \biggr\} \\
&\quad= \{ (h_0,h_1,h_2) \in U_1 \times U_2 \times U_3: h_{01} - 2h_{11} + h_{21} = 0 \}
\end{align*}
(with the convention $h_i = (h_{i1}, h_{i2}, h_{i3})$), and (for $p>3$)
\begin{align*}
&HP_{0,1,2,3}(U_1, U_2, U_3 )\\
&\quad = \biggl\{ \bigl((a_1,a_2,a_3), (a_1+b_1,a_2+b_2,a_3+b_3), (a_1+2b_1, a_2+2b_2+c_2, a_3+2b_3+c_3), \\
&\quad\quad (a_1+3b_1, a_2+3b_2+3c_2, a_3+3b_3+3c_3+d_3) \bigr) : \\
&\quad\quad\quad a_1,b_1 \in U_1; a_2,b_2,c_2 \in U_2; a_3,b_3,c_3,d_3 \in U_3 \biggr\} \\
&\quad= \{ (h_0,h_1,h_2) \in U_1 \times U_2 \times U_3: h_{01} - 2h_{11} + h_{21} = h_{11}-2h_{21}+h_{31}= 0,\\
&\quad\quad h_{02}-3h_{12}+3h_{22}-h_{32} = 0 \}.
\end{align*}

The following lemma, which we prove in Section \ref{limit-sec}, asserts that the Hall-Petresco group $HP_{c_0,\ldots,c_k}(\mathrm{Z}_{k-1})$ controls the equidistribution of progressions $(T_{c_0 g} x, \ldots, T_{c_k g} x)$ in $\mathrm{X}$:

\begin{lemma}[Limit formula]\label{equidist}  Let $p$ be a prime, let $1 \leq k < p$ be an integer, and let $c_0,\ldots,c_k \in \F_p$ be distinct.  Let $G := \F_p^\omega$, and let $\mathrm{X} = (X,\X,\mu,(T_g)_{g \in G})$ be an ergodic $G$-system.  Let $f_0,\ldots,f_k \in L^\infty(\mathrm{X})$, and let $(\Phi_n)_{n=1}^\infty$ be a F{\o}lner sequence in $G$. Then we have
\begin{equation}\label{formula}
 \lim_{n \to\infty} \E_{g \in \Phi_n} \int_X (T_{c_0} g f_0) \ldots (T_{c_kg} f_k)\ d\mu = \int_{HP_{c_0,\ldots,c_k}(\mathrm{Z}_{k-1})} (\pi_{k-1})_* f_0 \otimes\ldots \otimes (\pi_{k-1})_* f_k \ dm_{HP_{c_0,\ldots,c_k}(\mathrm{Z}_{k-1})}
\end{equation}
where $(\mathrm{Z}_{k-1}, \pi_{k-1})$ is the characteristic factor from Theorem \ref{char-thm}, $m_{HP_{c_0,\ldots,c_k}(\mathrm{Z}_{k-1})}$ is the Haar probability measure on $HP_{c_0,\ldots,c_k}(\mathrm{Z}_{k-1})$, and $(\pi_{k-1})_* f_0 \otimes \ldots \otimes (\pi_{k-1})_* f_k: \mathrm{Z}_{k-1}^{k+1} \to \C$ is the tensor product
$$ (\pi_{k-1})_*f_0 \otimes \ldots \otimes (\pi_{k-1})_*f_k(x_0,\ldots,x_k) := (\pi_{k-1})_* f_0(x_0) \ldots (\pi_{k-1})_* f_k(x_k).$$
The right-hand side of \eqref{formula} can also be written more explicitly as
$$ \int_{U_1^2 \times \ldots \times U_{k-1}^{k}} \prod_{i=0}^k (\pi_{k-1})_* f_i( ( \sum_{l=0}^j \binom{c_i}{l} a_{jl} )_{j=1}^{k-1} )$$
where the integral is over all tuples $(a_{jl})_{1 \leq j \leq k-1; 0 \leq l \leq j}$ with $a_{jl} \in U_j$, integrated using the product Haar measure on $U_1^2 \times \ldots \times U_{k-1}^{k}$, and $U_1,\ldots,U_{k-1}$ are the structure groups of $\mathrm{Z}_{k-1}$.
\end{lemma}

We remark that $HP_{c_0,\ldots,c_k}(\mathrm{Z}_{k-1})$ contains the diagonal group $\{ (x,\ldots,x): x \in Z_{k-1} \}$ and so surjects onto each of the $k+1$ coordinates of $(Z_{k-1})^{k+1}$.  In particular, the right-hand side of \eqref{formula} is well-defined even though each of the $f_i$ are only defined up to $\mu$-almost everywhere equivalence.

As examples of the formula \eqref{formula}, we have (for $p>2$)
$$ \lim_{n \to\infty} \E_{g \in \Phi_n} \int_X f_0 (T_g f_1) (T_{2g} f_2)\ d\mu = \int_{U_1^2} (\pi_1)_* f_0(x) (\pi_1)_* f_1(x+t)
(\pi_1)_* f_2(x+2t)\ dm_{U_1}(x) dm_{U_1}(t)$$
and (for $p > 3$)
\begin{align*}
& \lim_{n \to\infty} \E_{g \in \Phi_n} \int_X f_0 (T_g f_1) (T_{2g} f_2) (T_{3g} f_3)\ d\mu = \int_{U_1^2 \times U_2^3} \\
 &\quad (\pi_2)_* f_0(x_1, x_2) (\pi_2)_* f_1(x_1+t_1, x_2+t_2)
(\pi_2)_* f_2(x_1+2t_1, x_2+2t_2+u_2)\\
&\quad (\pi_2)_* f_3(x_1+3t_1, x_2+3t_2+3u_2) \\
&\quad\ dm_{U_1}(x_1) dm_{U_1}(t_1) dm_{U_2}(x_2) dm_{U_2}(t_2) dm_{U_2}(u_2).
\end{align*}

We also remark that if for $G=\Z$ one considers nilsystems instead of Weyl systems, the analogue of the Hall-Petresco group is the group of Hall-Petresco sequences \cite{hall}, \cite{petresco}, as can be seen from the equidistribution theory in \cite{leibman}. 

By duality, the above limit formula \eqref{formula} also gives a formula for limits (in $L^2(\mathrm{X})$) of the form
\begin{equation}\label{limt}
 \lim_{n \to\infty} \E_{g \in \Phi_n} (T_{c_1 g}  f_1) \ldots (T_{c_kg} f_k)\ d\mu;
\end{equation}
for instance, the limit
$$ \lim_{n \to\infty} \E_{g \in \Phi_n} (T_g f_1) (T_{2g} f_2)\ d\mu $$
in $L^2(\mathrm{X})$ is the function
$$
x \mapsto \int_{U_1} (\pi_1)_* f_1(\pi_1(x)+t) (\pi_1)_* f_2(\pi_1(x)+2t)\ dm_{U_1}(t).$$
The formula in the general case has a similar (but messier) appearance, and is omitted here.  In the above argument, we implicitly used the known result that the limit \eqref{limt} in $L^2(\mathrm{X})$ existed; but in fact the arguments in this paper give an independent proof of this norm convergence result, see Remark \ref{norm} below.

We also have an analogous limit formula for the correlation functions $I_{c_0,\ldots,c_k;f_0,\ldots,f_k}(g)$, which approximates these functions by a certain integral expression $J_{c_0,\ldots,c_k;f_0,\ldots,f_k}(g)$ up to a vanishingly small error in uniform density, in analogy to a similar result \cite[Proposition 6.5]{bhk} for $\Z$-systems.  To state this formula, we need some additional notation.   Let $G = \F_p^\omega$ for a prime $p$, let $1 \leq k < p$, and let $\mathrm{Z}_k$ be the characteristic factor from Theorem \ref{char-thm}, with structure groups $U_1,\ldots,U_k$. For each $1 \leq i \leq k$, let $u_i: Z_{k} \to U_i$ be the coordinate function.  Observe from the definition of a Weyl system that
$$ \Delta_g u_i = \rho_i \circ \pi_{i-1}$$
for all $1 \leq i \leq k$, where $\pi_{i-1}$ is the projection from $Z_{k-1}$ to $Z_{i-1}$.  In particular, as $\rho_i$ is a polynomial of degree $<i$, $u_i$ is a polynomial of degree $<i+1$.  For any $g\in G$, the derivative $\Delta_g^i u_i$ is a constant function, and can thus be identified with an element of $U_i$.  Given $g \in G$, we define the tuple $\theta(g) = (\theta_1(g),\ldots,\theta_k(g)) \in U_1 \times \ldots \times U_k$ by the formula
$$ \theta_i(g) := \Delta_g^i u_i.$$ 
We then define $HP_{c_0,\ldots,c_k}(\mathrm{Z}_{k})_{\theta(g)}$ to be the subset of $HP_{c_0,\ldots,c_k}(\mathrm{Z}_{k})$ consisting of  tuples $(P(c_0),\ldots,P(c_k))$ with $P = (P_1,\ldots,P_{k})$, where each $P_i: \Z \to U_i$ is a polynomial of degree $<i+1$ obeying the additional constraint
$$ \Delta_1^i P_i = \theta_i(g)$$
on the leading coefficient of each of the $P_i$.  
Note that $HP_{c_0,\ldots,c_k}(\mathrm{Z}_{k})_0$ is a closed subgroup of $HP_{c_0,\ldots,c_k}(\mathrm{Z}_{k})$, and $HP_{c_0,\ldots,c_k}(\mathrm{Z}_{k})_\theta$ is a coset of $HP_{c_0,\ldots,c_k}(\mathrm{Z}_{k})_0$ for any $\theta \in U_1 \times \ldots \times U_k$.  In particular, $HP_{c_0,\ldots,c_k}(\mathrm{Z}_{k})_{\theta(g)}$ has a well-defined Haar measure $dm_{HP_{c_0,\ldots,c_k}(\mathrm{Z}_{k})_{\theta(g)}}$.

\begin{lemma}[Second limit formula]\label{equidist-2}  Let $p$ be a prime, let $1 \leq k < p$ be an integer, and let $c_1,\ldots,c_k \in \F_p$ be distinct.  Let $G := \F_p^\omega$, and let $\mathrm{X} = (X,\X,\mu,(T_g)_{g \in G})$ be an ergodic $G$-system.  Let $f_0,\ldots,f_k \in L^\infty(\mathrm{X})$, and let $(\Phi_n)_{n=1}^\infty$ be a F{\o}lner sequence in $G$.   Define the sequence $J_{c_0,\ldots,c_k;f_1,\ldots,f_k}: G \to \C$ by the formula
$$ J_{c_0,\ldots,c_k;f_0,\ldots,f_k}(g) :=
\int_{HP_{c_0,\ldots,c_k}(\mathrm{Z}_{k})_\theta(g)} (\pi_k)_* f_0 \otimes\ldots \otimes (\pi_k)_* f_k \ dm_{HP_{c_0,\ldots,c_k}(\mathrm{Z}_{k})_{\theta(g)}}.$$
Then the difference $I_{c_0,\ldots,c_k;f_0,\ldots,f_k} - J_{c_0,\ldots,c_k;f_0,\ldots,f_k}$ converges to zero in uniform density.
\end{lemma}

Let us write $I(g) \approx_{UD} J(g)$ for the assertion that $I(g)-J(g)$ converges to zero in uniform density.  Then a simple special case of Lemma \ref{equidist-2} is the approximation
$$ \int_X f_0 T_g f_1\ d\mu \approx_{UD} \int_{U_1} (\pi_1)_* f_0( x ) (\pi_1)_* f_1( x + \Delta_g u_1 )\ dm_{U_1}(x);$$
similarly, we have (for $p>2$)
\begin{align*}
\int_X f_0 T_g f_1 T_{2g} f_2\ d\mu &\approx_{UD} \int_{U_1 \times U_2^2} (\pi_2)_* f_0( x_1, x_2 ) (\pi_2)_* f_1( x_1 + \Delta_g u_1, x_2 + t_2 )\\
&\quad (\pi_2)_* f_2( x_1 + 2\Delta_g u_1, x_2 + 2t_2 + \Delta_g^2 u_2 )\ dm_{U_1}(x_1) dm_{U_2}(x_2) dm_{U_2}(t_2)
\end{align*}
and (for $p>3$)
\begin{align*}
\int_X f_0 T_g f_1 T_{2g} f_2 T_{3g} f_3\ d\mu &\approx_{UD} \int_{U_1 \times U_2^2 \times U_3^3} (\pi_3)_* f_0( x_1, x_2, x_3 ) (\pi_3)_* f_1( x_1 + \Delta_g u_1, x_2 + t_2, x_3+t_3 )\\
&\quad (\pi_3)_* f_2( x_1 + 2\Delta_g u_1, x_2 + 2t_2 + \Delta_g^2 u_2, x_3 + 2t_3 + s_3 ) \\
&\quad (\pi_3)_* f_3( x_1 + 3\Delta_g u_1, x_2 + 3t_2 + 3\Delta_g^2 u_2, x_3 + 3t_3 + 3s_3 + \Delta_g^3 u_3 ) \\
&\quad\quad dm_{U_1}(x_1) dm_{U_2}(x_2) dm_{U_2}(t_2) dm_{U_3}(x_3) dm_{U_3}(t_3) dm_{U_3}(s_3).
\end{align*}

We prove this lemma in Section \ref{second-limit}.  The sequence $J_{c_0,\ldots,c_k;f_0,\ldots,f_k}$ can also be viewed as a ``Weyl sequence'' (analogous to the concept of a nilsequence, but with respect to a Weyl system rather than a nilsystem):

\begin{proposition}[Structure theorem]\label{struct}  Let the notation be as in Lemma \ref{equidist-2}.  Then there exists a continuous $k$-step Weyl system $\mathrm{Y} = (Y, \Y, \nu, (S_g)_{g \in G})$, a continuous $F \in L^\infty(\mathrm{Y})$, and a point $y_0 \in Y$ such that 
$$ J_{c_0,\ldots,c_k;f_0,\ldots,f_k}(g) = F( S_g y_0 ) $$
for all $g \in G$.
\end{proposition}

We establish this proposition in Section \ref{struct-sec}.  Combining this proposition with Lemma \ref{equidist-2}, we see that $I_{c_0,\ldots,c_k;f_0,\ldots,f_k}$ is approximated by a $k$-step Weyl sequence up to an error that goes to zero in uniform density (cf. \cite[Theorem 1.9]{bhk}).

In analogy to \cite{bhk}, we can use the limit formulae to obtain Khintchine type recurrence theorems.  It will be convenient to make the following definition.

\begin{definition}[Khintchine property]\label{khin-def}  Let $p$ be a prime, and let $c_0,\ldots,c_k$ be distinct elements of $\F_p$.  We say that the tuple $(c_0,\ldots,c_k)$ has the \emph{Khintchine property} (in characteristic $p$) if, whenever $G = \F_p^\omega$, $\mathrm{X} = (X,\X,\mu,(T_g)_{g \in G})$ is an ergodic $G$-system, $A \in \X$, and $\eps>0$, the set
$$ \{ g \in G: \mu( T_{-c_0 g} A \cap \ldots \cap T_{-c_k g} A ) \geq \mu(A)^{k+1} - \eps \}$$
is a syndetic subset of $G$ (i.e. $G$ can be covered by finitely many translates of this set).
\end{definition}

Of course, the negative signs in the subscripts here can be easily deleted if desired.  
It is trivial that any singleton tuple $(c_0)$ has the Khintchine property, and the classical Khintchine recurrence theorem adopted to general Abelian groups $G$ implies that any pair $(c_0,c_1)$ has the Khintchine property (and in this case we do not need to assume the ergodicity of our $G$-system).  It is also clear that the Khintchine property is preserved if one applies an invertible affine tranformation $x \mapsto ax+b$ to each element $c_i$ of a tuple $(c_0,\ldots,c_k)$ (i.e. $(c_0,\ldots,c_k)$ has the Khintchine property iff $(ac_0+b,\ldots,ac_k+b)$ has the Khintchine property).  For longer tuples, we have the following positive results (which are finite characteristic analogues of results in \cite{bhk}):

\begin{theorem}[Khintchine for double recurrence]\label{double}  If $p > 2$ and $c_0,c_1,c_2$ are distinct elements of $\F_p$, then $(c_0,c_1,c_2)$ has the Khintchine property.
\end{theorem}

\begin{theorem}[Khintchine for triple recurrence]\label{triple}  If $p > 3$ and $c_0,c_1,c_2,c_3$ are distinct elements of $\F_p$ which form a parallelogram in the sense that $c_i+c_j=c_k+c_l$ for some permutation $\{i,j,k,l\}$ of $\{1,2,3,4\}$, then $(c_0,c_1,c_2,c_3)$ has the Khintchine property.
\end{theorem}

For comparison, the classical Khintchine recurrence theorem, that is, a version of it for $G=\F_p^\omega$,  implies that for any distinct $c_0,c_1  \in \F_p$ and any $G$-system $\mathrm{X} = (X,\X,\mu,(T_g)_{g \in G})$, any $A \in \X$ and any $\eps> 0$, the set
$$ \{ g \in G: \mu( T_{-c_0 g} A \cap T_{-c_1g} A ) \geq \mu(A)^2 - \eps \}$$
is syndetic.  In this classical setting of single recurrence, no ergodicity hypothesis is required, but by adapting the construction in \cite[Theorem 2.1]{bhk}, one can show that ergodicity is needed for double or higher recurrence if $p$ is sufficiently large (this hypothesis is needed to embed a version of the Behrend-type constructions used in \cite{bhk}).

We prove these results in Section \ref{double-sec} and Section \ref{triple-sec} respectively. We remark that a finitary analogue of Theorem \ref{triple}, concerning dense subsets of finite-dimensional vector spaces $\F_p^n$ instead of subsets of $\F_p^\omega$-systems (and with the shifts $g$ lying in a dense subset of $\F_p^n$, rather than a syndetic subset of $\F_p^\omega$), was established in \cite[Theorem 4.1]{green}.

We conjecture that the above results exhaust all the possible tuples with the Khintchine property:

\begin{conjecture}  Let $p$ be a prime, let $k < p$, and let $c_0,\ldots,c_k$ be distinct elements of $\F_p$.
\begin{itemize}
\item[(i)]  If $k>3$, then $(c_0,\ldots,c_k)$ does not have the Khintchine property.
\item[(ii)]  If $k=3$, and $(c_0,c_1,c_2,c_3)$ does not form a parallelogram, then $(c_0,c_1,c_2,c_3)$ does not have the Khintchine property.
\end{itemize}
\end{conjecture}

In \cite[Corollary 1.6]{bhk}, it was shown that the tuple $(0,1,2,3,4)$ did not obey the analogous Khintchine property for $\Z$-systems, and it is not difficult to modify the construction there to also show that $(0,1,2,3,4)$ does not obey the Khintchine property in characteristic $p$ if $p$ is sufficiently large.  Similarly if $(0,1,2,3,4)$ is replaced by $(0,1,\ldots,k)$ for any fixed $k \geq 4$, if $p$ is sufficiently large depending on $k$.

While we were not able\footnote{By some extremely lengthy computations involving a subdivision into a large number of subcases, and \emph{ad hoc} constructions of counterexamples in each case, we have been able to verify this conjecture in the case when $k=3$, $c_0,c_1,c_2,c_3$ are fixed integers that do not form a parallelogram, and $p$ is sufficiently large depending on $c_0,c_1,c_2,c_3$ (or alternatively, if one considers $\Z$-systems rather than $\F_p^\omega$-systems).  We plan to make details of these constructions available elsewhere.} establish the above conjecture in general, we can do so for ``generic'' tuples $(c_0,\ldots,c_k)$:

\begin{theorem}[Khintchine property generically fails]\label{counter}  Let $k \geq 3$.  Then there exists a constant $C_k$ depending only on $k$ such that for any prime $p$, there are at most $C_k p^k$ tuples $(c_0,\ldots,c_k) \in \F_p^{k+1}$ that obey the Khintchine property.
\end{theorem}

In other words, if one selects $c_0,\ldots,c_k \in \F_p$ uniformly at random, then the Khintchine property will only hold with probability at most $C_k/p$, and so for large $p$ one has failure of the property for most tuples $(c_0,\ldots,c_k)$.

We establish this result in Section \ref{counter-sec}.  The question remains open as to whether a weakened version of the Khintchine property can hold in which $\mu(A)^{k+1}$ is replaced by a larger power $\mu(A)^{C_k}$ of $\mu(A)$.  In the case of $\Z$-systems, it was shown in \cite[Corollary 1.6]{bhk}  that this is not the case, at least for the model case $k=4$ and $c_i=i$.  However, this argument relies on the Behrend construction \cite{behrend}, and it remains an interesting open problem to usefully adapt this construction to the finite field setting when the characteristic $p$ is fixed.  Note though that it follows from the ``syndetic'' Szemer\'edi theorem for  vector spaces over finite fields \cite{furstenberg-katznelson} that the Khintchine property does hold if $\mu(A)^{k+1}$ is replaced by some sufficiently small quantity $c(k, \mu(A)) > 0$ depending only on $k$ and $\mu(A)$, if $\mu(A)$ is non-zero.

The first author is supported by support NSF grant DMS-1162073.  The second author is supported by NSF grant DMS-0649473 and by a Simons Investigator Award.  The third author is supported by ISF grant 407/12.  The third author was on sabbatical at  Stanford at the time this work was carried out and she would like to thanks the Stanford math department for its hospitality and support. We thank the anonymous referee for many useful corrections and suggestions.

\section{Continuity of polynomials}

In this section we establish a technical lemma that asserts, roughly speaking, that polynomials in an ergodic Weyl system are automatically continuous.  
\begin{lemma}[Polynomials are continuous]\label{polycon}  Let $G=(G,+)$ be a countable abelian group, let $k \geq 0$, and let $\mathrm{X} = (X,\X,\mu,(T_g)_{g \in G})$ be an ergodic $k$-step Weyl system.
\begin{itemize}
\item[(i)]  After modifying the cocycles used to define $\mathrm{X}$ on a measure zero set if necessary, $\mathrm{X}$ becomes a continuous $k$-step Weyl system.
\item[(ii)]  If $\phi: X \to \R/\Z$ is a polynomial, then (after redefining $\phi$ on a measure zero set if necessary), $\phi$ is continuous.
\end{itemize}
\end{lemma}

\begin{proof}  We induct on $k$.  The $k=0$ case is trivial, so suppose that $k \geq 1$ and the claims (i), (ii) have already been proven for all smaller values of $k$.  The claim (i) for $k$ then follows by applying the induction hypothesis (ii) to all the cocycles used to construct $\mathrm{X}$, so now we turn to claim (ii) for $k$.  Write $\mathrm{X} = \mathrm{X}_{k-1} \times_{\rho_k} U_k$ for some compact abelian $U_k$ and some polynomial $(G,\mathrm{X}_{k-1},U_k)$ cocycle $\rho_k$ of degree $<k$.  By claim (i) for $k$, we may assume without loss of generality that all cocycles involved in constructing $\mathrm{X}$ (including $\rho_k$) are continuous.  

Let us first handle the case when the compact group $U_k$ is finite (and thus discrete).  For each $u_k \in U_k$, the function $\phi_{u_k}: X_{k-1} \to \R/\Z$ defined by
$$ \phi_{u_k}(x_{k-1}) := \phi(x_{k-1},u_k)$$
is a polynomial on $\mathrm{X}_{k-1}$ (see \cite[Lemma B.5(iii)]{bergelson-tao-ziegler}), and can thus be modified on a set of zero to become continuous.  Applying this for each $u_k$ and gluing, we obtain the claim.

Now we handle the general case when $U_k$ is not necessarily finite.  For every $t \in U_k$, define the \emph{vertical derivative} $\Delta_t \phi: X \to \R/\Z$ of $\phi$ by the formula
$$ \Delta_t \phi( x_{k-1}, u_k ) := \phi(x_{k-1}, u_k+t) - \phi(x_{k-1},u_k).$$
As $\phi$ is a polynomial, $\Delta_t \phi$ is a polynomial of uniformly bounded degree.  On the other hand, as $\phi$ is measurable, $\Delta_t \phi$ converges to zero in measure as $t \to 0$ in $U_k$, and in particular
$$ \| e(\Delta_t \phi) - 1 \|_{L^2(\mathrm{X})} \to 0$$
as $t \to 0$, where $e(x) := e^{2\pi ix}$ is the standard character on $\R/\Z$.   By \cite[Lemma C.1]{bergelson-tao-ziegler} and the uniformly bounded degree of $\Delta_t \phi$, we conclude that $e(\Delta_t \phi)$ must be almost everywhere constant for $t$ sufficiently close to $0$.  In particular, for $t$ sufficiently close to zero there exists $\chi(t) \in \R/\Z$ such that $\Delta_t \phi = \chi(t)$ almost everywhere.  The set of all $t$ with this property is easily seen to form a compact open subgroup $U'_k$ of $U_k$, and $\chi$ is a homomorphism from $U'_k$ to $\R/\Z$ which is measurable, and hence continuous (Steinhaus lemma).  By Pontryagin duality, $\chi$ can then be extended to an additive character from $U_k$ to $\R/\Z$.  If we then define the function $\psi: X \to \R/\Z$ by
$$ \psi( x_{k-1}, u_k ) := \phi(x_{k-1}, u_k) - \chi(u_k)$$
then we see that $\psi$ is also a polynomial on $\mathrm{X}$, with $\Delta_t \psi = 0$ for $t \in U'_k$; thus, after modification on a set of measure zero, $\psi$ is constant on all $U'_k$-orbits.  We can then quotient $U_k$ and $\rho_k$ by $U'_k$ and reduce to a quotiented Weyl system $\mathrm{X}_{k-1} \times_{\rho_k \hbox{ mod } U'_k} U_k/U'_k$ in which the final group $U_k/U'_k$ is finite.  By the case already treated, the quotiented version of $\psi$ can be modified on a set of measure zero to become continuous on this system.  Thus, on applying pullback, the same claim holds for $\psi$ and hence for $\phi$, giving the claim. 
\end{proof}

\section{Gowers-Host-Kra seminorms and characteristic factors}\label{char-sec}

In this section we derive Theorem \ref{char-thm} from the theory of Gowers-Host-Kra seminorms on $\F_p^\omega$-systems as developed in \cite{bergelson-tao-ziegler}.  The material here is very standard for $\Z$-systems (see \cite{host-kra}), and the adaptation of that theory to $\F_p^\omega$-systems is routine, but we present it here for the sake of completeness.

We first recall the definition of the Gowers-Host-Kra seminorms, introduced in \cite{host-kra} (and closely related to the combinatorial Gowers uniformity norms from \cite{gowers}):

\begin{definition}[Gowers-Host-Kra seminorms]\cite{host-kra}  Let $G = (G,+)$ be a countable abelian group, and let $\mathrm{X} = (X, \X, \mu, (T_g)_{g \in G})$ be a $G$-system.  For any $f \in L^\infty(\mathrm{X})$, we define the \emph{Gowers-Host-Kra seminorms} $\|f\|_{U^k(\mathrm{X})}$ recursively for $k \geq 1$ by setting
\begin{equation}\label{u1-def}
 \|f\|_{U^1(\mathrm{X})} := \lim_{n \to \infty} \| \E_{h \in \Phi_n} T_h f \|_{L^2(\mathrm{X})}
\end{equation}
and
\begin{equation}\label{uk-def}
 \|f\|_{U^k(\mathrm{X})} := \left(\lim_{n \to \infty} \| \E_{h \in \Phi_n} T_h f \overline{f} \|_{U^{k-1}(\mathrm{X})}^{2^{k-1}}\right)^{1/2^k}
 \end{equation}
for any $k \geq 2$ and any F{\o}lner sequence $(\Phi_n)_{n=1}^\infty$ of $G$.
\end{definition}

One can show that the above definitions are in fact independent of the choice of the F{\o}lner sequence, and define a sequence of seminorms on $L^\infty(\mathrm{X})$; see\footnote{Strictly speaking, in that lemma the additional hypothesis of nestedness $\Phi_1 \subset \Phi_2 \subset \ldots$ of the F{\o}lner sequence is imposed, but an inspection of the proof shows that this hypothesis is not needed (because the mean ergodic theorem holds for non-nested F{\o}lner sequences).} \cite[Lemma A.18]{bergelson-tao-ziegler}.  From the mean ergodic theorem, we observe that the $U^1$ seminorm can also be written as
\begin{equation}\label{lo}
\|f\|_{U^1(\mathrm{X})} = \| (\pi_0)_* f \|_{L^2(\mathrm{Z}_0)}
\end{equation}
where $(\mathrm{Z}_0,\pi_0)$ is the invariant factor.

The significance of these seminorms for us is that they control the convergence of expressions such as \eqref{tgf-mult}.  More precisely, we have the following minor variant of \cite[Theorem 12.1]{host-kra}:

\begin{lemma}[Generalized van der Corput lemma]\label{vdc}  Let $G = \F_p^\omega$ for a prime $p$, and let $\mathrm{X} = (X, \X, \mu, (T_g)_{g \in G})$ be a $G$-system.  Let $1 \leq k < p$, and let $c_1,\ldots,c_k$ be distinct elements of $\F_p \backslash \{0\}$.  Let $f_1,\ldots,f_k \in L^\infty(\mathrm{X})$.  Let $(\Phi_n)_{n=1}^\infty$ be a F{\o}lner sequence of $G$.  Then we have
$$ \limsup_{n \to \infty} \| \E_{g \in \Phi_n} (T_{c_1g} f_1) (T_{c_2g} f_2) \ldots (T_{c_kg} f_k) \|_{L^2(\mathrm{X})} \leq 
\inf_{1 \leq i \leq k} \| f_i \|_{U^k(\mathrm{X})} \prod_{1 \leq j \leq k:j \neq i} \|f_j\|_{L^\infty(\mathrm{X})}.$$
\end{lemma}

\begin{proof} We induct on $k$.  When $k=1$, we may rescale $\Phi_n$ by $c_1$ to normalize $c_1=1$, and then the claim follows from \eqref{u1-def}.  Now suppose that $k > 1$, and that the claim has already been proven for $k-1$.  By permuting indices, it suffices to show that
\begin{equation}\label{limsup}
 \limsup_{n \to \infty} \| \E_{g \in \Phi_n} F_g \|_{L^2(\mathrm{X})} \leq 
\| f_k \|_{U^k(\mathrm{X})}
\end{equation}
where $\|f_i\|_{L^\infty(\mathrm{X})} \leq 1$ for $i=1,\ldots,k-1$, and
$$ F_g := (T_{c_1g} f_1) (T_{c_2g} f_2) \ldots (T_{c_kg} f_k).$$
We may also normalize $c_k=1$.
Using the F{\o}lner property, we can rewrite the left-hand side of \eqref{limsup} as
$$
\limsup_{m \to \infty} \limsup_{n \to \infty} \| \E_{g \in \Phi_n} \E_{h \in \Phi_m} F_{g+h} \|_{L^2(\mathrm{X})}$$
which we can then bound using the triangle inequality by
$$
\limsup_{m \to \infty} \limsup_{n \to \infty} \E_{g \in \Phi_n} \| \E_{h \in \Phi_m} F_{g+h} \|_{L^2(\mathrm{X})}.$$
By Cauchy-Schwarz, this is bounded by
$$
\left(\limsup_{m \to \infty} \limsup_{n \to \infty} \E_{g \in \Phi_n} \| \E_{h \in \Phi_m} F_{g+h} \|_{L^2(\mathrm{X})}^2\right)^{1/2}$$
which we may expand as
$$
\left(\limsup_{m \to \infty} \limsup_{n \to \infty} \E_{h,h' \in \Phi_m} \int_X \E_{g \in \Phi_n} F_{g+h} \overline{F_{g+h'}}\ d\mu\right)^{1/2}.$$
We may upper bound this expression by
$$
\left(\limsup_{m \to \infty} \E_{h,h' \in \Phi_m} \limsup_{n \to \infty} |\int_X \E_{g \in \Phi_n} F_{g+h} \overline{F_{g+h'}}|\ d\mu\right)^{1/2}.$$
Now, for each $h,h'$ we may write
$$
\int_X \E_{g \in \Phi_n} F_{g+h} \overline{F_{g+h'}} = \int_X (T_{c_1 h} f_1) \overline{T_{c_1 h'} f_1} \E_{g \in \Phi_n} \prod_{i=2}^k T_{(c_i-c_1)g} ((T_{c_i h} f_i) \overline{T_{c_i h'} f_i})\ d\mu.$$
Applying Cauchy-Schwarz and the induction hypothesis (and the normalization $c_k=1$), we conclude that
$$ \limsup_{n \to \infty} |\int_X \E_{g \in \Phi_n} F_{g+h} \overline{F_{g+h'}}|\ d\mu \leq \| (T_{h} f_k) \overline{T_{h'} f_k} \|_{U^{k-1}(\mathrm{X})}.$$
Note that $(T_h f_k) \overline{T_{h'} f_k}$ has the same $U^{k-1}(\mathrm{X})$ norm as $(T_{h-h'} f_k) \overline{f_k}$.
Putting all this together, we can bound the left-hand side of \eqref{limsup} by
$$
\left(\limsup_{m \to \infty} \E_{h,h' \in \Phi_m} \| T_{h-h'} f_k \overline{f_k} \|_{U^{k-1}(\mathrm{X})}\right)^{1/2}.
$$
By the triangle inequality and the pigeonhole principle, we may bound this by
$$
(\limsup_{m \to \infty} \E_{h \in \Phi_m - h'_m} \| T_h f_k \overline{f_k} \|_{U^{k-1}(\mathrm{X})})^{1/2}
$$
for some sequence $h'_m \in \Phi_m$; by H\"older's inequality, we may bound this by
$$
(\limsup_{m \to \infty} \E_{h \in \Phi_m - h'_m} \| T_h f_k \overline{f_k} \|_{U^{k-1}(\mathrm{X})}^{2^{k-1}})^{1/2^k}.
$$
But the $\Phi_m-h'_m$ form a F{\o}lner sequence of $G$, and the claim \eqref{limsup} then follows from \eqref{uk-def}.
\end{proof}

We also need the following variant:

\begin{lemma}[Generalized van der Corput lemma, II]\label{vdc-2}  Let $G = \F_p^\omega$ for a prime $p$, and let $\mathrm{X} = (X, \X, \mu, (T_g)_{g \in G})$ be a $G$-system.  Let $1 \leq k < p$, and let $c_0, c_1,\ldots,c_k$ be distinct elements of $\F_p$.  Let $f_0,\ldots,f_k \in L^\infty(\mathrm{X})$.  Let $(\Phi_n)_{n=1}^\infty$ be a F{\o}lner sequence of $G$.  Then we have
$$ \limsup_{n \to \infty} \sup_{h \in G} \E_{g \in h+\Phi_n} |I_{c_0,\ldots,c_k;f_0,\ldots,f_k}(g)|
\leq \inf_{0 \leq i \leq k} \| f_i \|_{U^{k+1}(\mathrm{X})} \prod_{0 \leq j \leq k:j \neq i} \|f_j\|_{L^\infty(\mathrm{X})}.$$
\end{lemma}

\begin{proof}  As before, it suffices to show that
$$ \limsup_{n \to \infty} \sup_{h \in G} \E_{g \in h+\Phi_n} |I_{c_0,\ldots,c_k;f_0,\ldots,f_k}(g)|
\leq \|f_k\|_{U^{k+1}(\mathrm{X})}$$
under the normalization $\|f_j\|_{L^\infty(\mathrm{X})} \leq 1$ for $0 \leq j < k$.  

Next, we remove the supremum in the above estimate.  Suppose that we can already show that
\begin{equation}\label{lsup}
\limsup_{n \to \infty} \E_{g \in \Phi_n} |I_{c_0,\ldots,c_k;f_0,\ldots,f_k}(g)|
\leq \|f_k\|_{U^{k+1}(\mathrm{X})}.
\end{equation}
For any $\eps>0$, and any $n$, we can find $h_n = h_{n,\eps} \in G$ such that
$$ \sup_{h \in G} \E_{g \in h+\Phi_n} |I_{c_0,\ldots,c_k;f_0,\ldots,f_k}(g)| \leq (1+\eps) 
\E_{g \in h_n+\Phi_n} |I_{c_0,\ldots,c_k;f_0,\ldots,f_k}(g)|.$$
Applying \eqref{lsup} to the F{\o}lner sequence $(h_n+\Phi_n)_{n=1}^\infty$ we conclude that
$$ \limsup_{n \to \infty} \sup_{h \in G} \E_{g \in h+\Phi_n} |I_{c_0,\ldots,c_k;f_0,\ldots,f_k}(g)|
\leq (1+\eps) \|f_k\|_{U^{k+1}(\mathrm{X})}$$
and the claim then follows by sending $\eps$ to zero.

It remains to establish \eqref{lsup}.  Write $F_g := (T_{c_0g} f_0) \ldots (T_{c_kg} f_k)$, so that $I_{c_0,\ldots,c_k;f_0,\ldots,f_{k}}(g) = \int_X F_g\ d\mu$.  For any $m$, we thus have
$$ |I_{c_0,\ldots,c_k;f_0,\ldots,f_{k}}(g)| = \int_X \E_{h \in \Phi_m} T_h F_g\ d\mu$$
and hence, by Cauchy-Schwarz, the left-hand side of \eqref{lsup} is bounded by
$$
(\limsup_{m \to \infty}
 \limsup_{n \to \infty} \E_{g \in \Phi_n} \int_X |\E_{h \in \Phi_m} T_h F_g|^2\ d\mu)^{1/2}.$$
We can expand this expression as
$$
(\limsup_{m \to \infty}
 \limsup_{n \to \infty} \E_{h, h' \in \Phi_m} \E_{g \in \Phi_n} \int_X (T_h F_g) (T_{h'} \overline{F_g})\ d\mu)^{1/2}$$
which by the triangle inequality is bounded by
$$(\limsup_{m \to \infty}\E_{h, h' \in \Phi_m} 
 |\limsup_{n \to \infty} \E_{g \in \Phi_n} \int_X (T_h F_g) (T_{h'} \overline{F_g})\ d\mu|)^{1/2}.$$
We can rewrite $\int_X (T_h F_g) (T_{h'} \overline{F_g})\ d\mu$ as
$$ \int_X \prod_{j=0}^k T_{c_jg} ( (T_h f_j) \overline{T_{h'} f_j} )\ d\mu.$$
Applying Lemma \ref{vdc}, we may thus bound the left-hand side of \eqref{lsup} by
$$
(\limsup_{m \to \infty}\E_{h, h' \in \Phi_m} \| (T_h f_k) \overline{T_{h'} f_k} \|_{U^{k}(\mathrm{X})})^{1/2}.$$
One can then argue as in the proof of Lemma \ref{vdc} to bound this by $\|f_k\|_{U^{k+1}(\mathrm{X})}$ as required.
\end{proof}

\begin{corollary}\label{gvn-cor}  Let $G = \F_p^\omega$ for a prime $p$, and let $\mathrm{X} = (X, \X, \mu, (T_g)_{g \in G})$ be a $G$-system.  Let $1 \leq k < p$.  Let $f_0,f_1,\ldots,f_k \in L^\infty(\mathrm{X})$.  Let $(\Phi_n)_{n=1}^\infty$ be a F{\o}lner sequence of $G$.  
\begin{itemize}
\item[(i)]  If $c_1,\ldots,c_k$ are distinct nonzero elements of $\F_p$, then the sequence
$$ \E_{g \in \Phi_n} (T_{c_1 g} f_1) \ldots (T_{c_kg} f_k)$$
converges in $L^2(\mathrm{X})$ to zero whenever $\|f_i\|_{U^k(\mathrm{X})}=0$ for some $1 \leq i \leq k$.
\item[(ii)]  If  $c_0,\ldots,c_k$ are distinct elements of $\F_p$, the sequence
$$ \E_{g \in \Phi_n} \int_X (T_{c_0g} f_0) \ldots (T_{c_kg} f_k) $$
converges to zero whenever $\|f_i\|_{U^k(\mathrm{X})}=0$ for some $0 \leq i \leq k$.
\item[(iii)]  If $c_0,\ldots,c_k$ are distinct elements of $\F_p$, the sequence $I_{c_0,\ldots,c_k;f_0,\ldots,f_{k}}(g)$ converges in uniform density to zero whenever $\|f_i\|_{U^{k+1}(\mathrm{X})}=0$ for some $0 \leq i \leq k$.
\end{itemize}
\end{corollary}

\begin{proof}  The claim (i) is immediate from Lemma \ref{vdc}.  To prove (ii), we may first permute so that $\|f_k\|_{U^k(\mathrm{X})}=0$, and then translate so that $c_0=0$.  The claim then follows from (i) after using Cauchy-Schwarz to eliminate $f_0$.  Finally, (iii) follows from Lemma \ref{vdc-2}.
\end{proof}

We remark that one can also prove (iii) using the structure of Host-Kra measures, after performing an ergodic decomposition; see \cite[Corollary 4.5]{host-kra}.

Theorem \ref{char-thm} is then immediate from Corollary \ref{gvn-cor} and the following result from \cite{bergelson-tao-ziegler}.

\begin{theorem}[Characteristic factor for the $U^k$ norm]\label{char-uk} Let $G = \F_p^\omega$ for a prime $p$, and let $\mathrm{X} = (X, \X, \mu, (T_g)_{g \in G})$ be an ergodic $G$-system.  For each $1 \leq k \leq p$, let $\B_{<k}$ be the sub-$\sigma$-algebra of $\X$ generated by the polynomials $\phi: X \to \R/\Z$ of degree $<k$.  Then there is a factor $(\mathrm{Z}_{k-1}, \pi_{k-1})$ of $\mathrm{X}$ that is equivalent to $\B_{<k}$, and there is a continuous ergodic $k-1$-step Weyl system, with $\mathrm{Z}_k = \mathrm{Z}_{k-1} \times_{\rho_k} U_k$ for all $1 \leq k < p$, some compact abelian $p$-torsion group $U_k$, and some polynomial $(G,\mathrm{Z}_{k-1},U_k)$-cocycle $\rho_k$ of degree $<k$.  Furthermore, if $f \in L^\infty(\mathrm{X})$, then $\|f\|_{U^k(\mathrm{X})} = 0$ if and only if $(\pi_{k-1})_* f = 0$.
\end{theorem}

\begin{proof}  This follows from \cite[Proposition 1.10]{bergelson-tao-ziegler}, \cite[Theorem 1.19]{bergelson-tao-ziegler} and \cite[Corollary 8.7]{bergelson-tao-ziegler}, using Lemma \ref{polycon} to ensure that the Weyl system obtained is continuous.  The ergodicity of the Weyl systems is automatic because any factor of an ergodic system is again ergodic.
\end{proof}

\begin{remark}  The condition $k \leq p$ was subsequently removed in \cite{tz-low} (but with the important caveat that the groups $U_j$ need no longer be $p$-torsion, but are merely $p^m$-torsion for some $m \geq 1$); however for our application we have $k \leq p$, so we will not need the (more difficult) arguments from \cite{tz-low} here.  It is also possible that the arguments in \cite{szegedy} could be adapted to give an alternate proof of Theorem \ref{char-uk}, although we will not pursue this approach here.
\end{remark}

\section{Some special cases of the limit formulae}

In the next two sections, we will prove the main limit formulae, namely Lemma \ref{equidist} and Lemma \ref{equidist-2}.  In order to motivate the proof of these formulae in the general case, we will discuss some model cases of these results here.

We begin with a special case of Lemma \ref{equidist}, when $p>2$ and $\mathrm{X}$ is a $2$-step Weyl system $\mathrm{X} = U_1 \times_{\rho_2} U_2$, where $\rho_2$ is a polynomial $(G, U_1, U_2)$-cocycle of degree $<2$.  Furthermore we assume (abusing notation slightly) that the base system $U_1$ is also the Kronecker factor $\mathrm{Z}_1(\mathrm{X})$, thus the only polynomials of degree $<2$ on $\mathrm{X}$ are those which are functions of the $U_1$ coordinate only.  The special case of Lemma \ref{equidist} we will discuss is
\begin{equation}\label{test}
\begin{split}
& \lim_{n \to\infty} \E_{g \in \Phi_n} \int_X f_0 (T_g f_1) (T_{2g} f_2) (T_{3g} f_3)\ d\mu = \int_{U_1^2 \times U_2^3} \\
 &\quad f_0(x_1, x_2) f_1(x_1+t_1, x_2+t_2)
f_2(x_1+2t_1, x_2+2t_2+u_2)
f_3(x_1+3t_1, x_2+3t_2+3u_2) \\
&\quad\ dm_{U_1}(x_1) dm_{U_1}(t_1) dm_{U_2}(x_2) dm_{U_2}(t_2) dm_{U_2}(u_2).
\end{split}
\end{equation}
(The factor map $\pi_2$ is not needed in this special case, as it is the identity map.)  To simplify things further, we assume that each function $f_i$ for $i=0,1,2,3$ takes the special form
\begin{equation}\label{fii}
 f_i(x_1,x_2) = e( \phi_{i,2}(x_2) )
\end{equation}
for some additive character (i.e. continuous homomorphism) $\phi_{i,2}: U_2 \to \R/\Z$.  One can (and should) also consider the slightly more general example of functions of the form
\begin{equation}\label{fii-2}
 f_i(x_1,x_2) = e( \phi_{i,1}(x_1) + \phi_{i,2}(x_2) )
\end{equation}
where $\phi_{i,1}: U_1 \to \R/\Z$ is an additive character of $U_1$, as these Fourier-analytic examples then span a dense subspace of $L^2(\mathrm{X})$, but to keep the discussion here simple, we will ignore the lower order terms $\phi_{i,1}$ and focus only on the examples of the form \eqref{fii}.

The verification of \eqref{test} now splits into several cases, depending on the nature of the characters $\phi_{0,2}, \ldots, \phi_{3,2}$.  One easy case is when $\phi_{0,2},\ldots,\phi_{3,2}$ all vanish identically; then both sides of \eqref{test} are clearly equal to $1$.

Next, suppose that $\phi_{3,2}$ vanishes identically, but one of the other $\phi_{i,2}$ is not identically zero; let's say for sake of concreteness that $\phi_{2,2}$ is not identically zero.  Then our task is to show that
\begin{equation}\label{sailor}
\begin{split}
 \lim_{n \to\infty} \E_{g \in \Phi_n} \int_X f_0 (T_g f_1) (T_{2g} f_2) \ d\mu &= \int_{U_2^3} 
  e( \phi_{0,2}(x_2) + \phi_{1,2}(x_2+t_2) + 
 \phi_{2,2}(x_2+2t_2+u_2) )\\ 
 &\qquad\ dm_{U_2}(x_2) dm_{U_2}(t_2) dm_{U_2}(u_2).
\end{split}
\end{equation}
Observe that $(x_2,t_2,u_2)$ varies over $U_2^3$, the tuple $(x_2,x_2+t_2,x_2+2t_2+u_2)$ is unconstrained in $U_2^3$.  In particular, if $(x_2,t_2,u_2)$ is drawn uniformly at random using the Haar measure on $U_2^3$, then $(x_2,x_2+t_2,x_2+2t_2+u_2)$ is also uniformly distributed with this Haar measure.  Thus the right-hand side factors as
$$ \left(\int_{U_2} e(\phi_{0,2})\ d\mu_2\right) \left(\int_{U_2} e(\phi_{1,2})\ d\mu_2\right) \left(\int_{U_2} e(\phi_{2,2})\ d\mu_2\right).$$
By Fourier analysis, the third factor vanishes since $\phi_{2,2}$ is assumed to not be identically zero, so the right-hand side of \eqref{sailor} vanishes.  As for the left-hand side, observe that the function $f_2(x_1,x_2) = e( \phi_{2,2}(x_2))$ has mean zero on every coset of $U_2$ in $U_1 \times U_2$; since we are assuming $U_1 = \mathrm{Z}_1$, this implies that $(\pi_1)_* f_2 = 0$.  By Theorem \ref{char-uk}, this implies that
$$ \|f_2\|_{U^2(\mathrm{X})} = 0.$$
Applying Corollary \ref{gvn-cor}, we conclude that the left-hand side of \eqref{sailor} vanishes also, so we are done in this case.

Finally, we consider the case when $\phi_{3,2}$ does not vanish identically.  We can then simplify the right-hand side of \eqref{test} by noting the extrapolation identity
$$ x_2+3t_2+3u_2 = x_2 - 3 (x_2 + t_2) + 3 (x_2 + 2t_2 + u_2)$$
which allows us to write the right-hand side as
$$
\int_{U_2^3} e\left( \phi'_{0,2}(x_2) + \phi'_{1,2}(x_2+t_2) + 
 \phi'_{2,2}(x_2+2t_2+u_2)\right)\ dm_{U_2}(x_2) dm_{U_2}(t_2) dm_{U_2}(u_2),$$
where $\phi'_{0,2} := \phi_{0,2} + \phi_{3,2}$, $\phi'_{1,2} := \phi_{1,2} - 3\phi_{3,2}$, and $\phi'_{2,2} := \phi_{2,2} + 3 \phi_{3,2}$.  Next, observe that
$$ \Delta_g u_2 = \rho(g,\cdot)$$
for all $g \in G$; since $\rho$ is a polynomial cocycle of degree $<2$, we conclude that $u_2$ is a polynomial of degree $<3$ (i.e. a quadratic function).  For any $g \in G$ and $x \in X$, the sequence $n \mapsto u_2( T_{ng} x )$ is then also a quadratic polynomial.  In particular, we have the interpolation identity
$$ u_2( T_{3g} x ) = u_2(x) - 3 u_2(T_g x) + 3 u_2(T_{2g} x)$$
which allows one to write the left-hand side of \eqref{test} as
$$ \lim_{n \to\infty} \E_{g \in \Phi_n} \int_X f'_0 (T_g f'_1) (T_{2g} f'_2)\ d\mu $$
where $f'_i(x_1,x_2) := e( \phi'_{i,2}(x_2) )$ for $i=0,1,2$.  As such, we see that we have reduced the task of verifying \eqref{test} when $\phi_{3,2}$ does not vanish identically, to the task of verifying \eqref{test} when $\phi_{3,2}$ does vanish identically.  But this case has already been covered by the preceding arguments.  This concludes the demonstration of \eqref{test} for the model functions \eqref{fii}.  The model cases
\eqref{fii-2} can be handled by similar arguments, exploiting the linear nature of $n \mapsto u_1( T_{ng} x)$ in addition to the quadratic nature of $n \mapsto u_2(T_{ng} x)$ to eventually reduce to the case when $\phi_{2,1}, \phi_{3,1}, \phi_{3,2}$ vanish and $\phi_{2,2}$ does not vanish identically, which can then be treated by Theorem \ref{char-uk} and Corollary \ref{gvn-cor} as before; we leave the details to the interested reader (and they are special cases of the argument in Section \ref{limit-sec} below).

Now we consider an analogous example for the second limit formula, Lemma \ref{equidist-2}.  Keeping the system $\mathrm{X} = U_1 \times_{\rho_2} U_2$ as before, we now consider the task of showing that
\begin{equation}\label{test-2}
\begin{split}
\int_X f_0 T_g f_1 T_{2g} f_2 T_{3g} f_3\ d\mu &\approx_{UD} \int_{U_1 \times U_2^2} f_0( x_1, x_2 ) f_1( x_1 + \Delta_g u_1, x_2 + t_2 )\\
&\quad f_2( x_1 + 2\Delta_g u_1, x_2 + 2t_2 + \Delta_g^2 u_2 ) f_3( x_1 + 3\Delta_g u_1, x_2 + 3t_2 + 3\Delta_g^2 u_2 )  \\
&\quad\quad dm_{U_1}(x_1) dm_{U_2}(x_2) dm_{U_2}(t_2),
\end{split}
\end{equation}
where we are using the $\approx_{UD}$ notation from the introduction.

Again, we will restrict attention to the model case \eqref{fii} for simplicity.  If the $\phi_{i,2}$ all vanish identically, then the claim is trivial as before.  Now suppose that $\phi_{2,2}$ and $\phi_{3,2}$ both vanish identically, but $\phi_{1,2}$ does not vanish identically.  The right-hand side of \eqref{test-2} then simplifies to
$$ \int_{U_2^2} e( \phi_{0,2}(x_2) + \phi_{1,2}(x_2+t_2) )\ dm_{U_2}(x_2) dm_{U_2}(t_2)$$
which vanishes by a change of variables and Fourier analysis.  Meanwhile, the left-hand side of \eqref{test-2} is $\int_X f_0 T_g f_1\ d\mu$; the non-vanishing of $\phi_{1,2}$ guarantees that $\|f_1\|_{U^2(\mathrm{X})}=0$ by Theorem \ref{char-uk}, and so by Corollary \ref{gvn-cor} the left-hand side goes to zero in uniform density, as required.

Now suppose that $\phi_{3,2}$ vanishes identically, but $\phi_{2,2}$ does not.  For any $g \in G$ and $x \in X$, we consider the sequence $\psi_{2,g,x}: \Z \to U_2$ defined by $\psi_{2,g,x}(n) := u_2( T_{ng} x )$.  As discussed earlier in this section, $\psi_{2,g,x}$ is a quadratic sequence.  However, for fixed $g$, we can also compute the top order coefficient $\Delta_1^2 \psi_{2,g,x}$ of this sequence:
$$ \Delta_1^2 \psi_{2,g,x} = \Delta_g^2 u_2.$$
Note that as $\psi_{2,g,x}$ and $u_2$ are both quadratic, the left and right-hand sides here are constants (i.e. elements of $U_2$).  Thus, $\psi_{2,g,x}$ is not an arbitrary quadratic sequence, but is in fact the sum of the sequence $n \mapsto \binom{n}{2} \Delta_g^2 u_2$ and a linear sequence.  This leads to the need to use an additional Lagrange interpolation formula
$$ \psi_{2,g,x}(2) = - \psi_{2,g,x}(0) +2 \psi_{2,g,x}(1) + \Delta_g^2 u_2$$
which allows one to rewrite the left-hand side of \eqref{test-2} as
$$ e( \Delta_g^2 u_2 ) \int_X f'_0 T_g f'_1\ d\mu$$
where $f'_i := e(\phi'_{i,2})$ for $i=1,2$, $\phi'_{0,2} := \phi_{0,2} - \phi_{2,2}$, and $\phi'_{1,2} := \phi_{1,2} + 2 \phi_{2,2}$.  Similarly, with the help of  the identity
$$ x_2 + 2t_2 + \Delta_g^2 u_2 = - x_2 + 2 (x_2+t_2) + \Delta_g^2 u_2$$
we may rewrite the right-hand side of \eqref{test-2} as
$$ e( \Delta_g^2 u_2 ) \int_{U_1 \times U_2^2} f'_0( x_1, x_2 ) f'_1( x_1 + \Delta_g u_1, x_2 + t_2 )
\ dm_{U_1}(x_1) dm_{U_2}(x_2) dm_{U_2}(t_2).$$
From the previously handled cases of \eqref{fii}, we already know that
$$ \int_X f'_0 T_g f'_1\ d\mu \approx_{UD} \int_{U_1 \times U_2^2} f'_0( x_1, x_2 ) f'_1( x_1 + \Delta_g u_1, x_2 + t_2 )
 \  dm_{U_1}(x_1) dm_{U_2}(x_2) dm_{U_2}(t_2).$$
Multiplying through by the phase  $e( \Delta_g^2 u_2 )$, we obtain \eqref{test-2} in the case that $\phi_{3,2}$ vanishes, but $\phi_{2,2}$ does not necessarily vanish.  A similar calculation (which we omit) then allows one to extend the previous cases to also cover the case when $\phi_{3,2}$ does not necessarily vanish either, giving \eqref{test-2} in all instances of the model case \eqref{fii}.  Again, the addition of the lower order terms in \eqref{fii-2} can be handled by a modification of these arguments, which we leave to the reader (and are special cases of the argument in Section \ref{second-limit} below).

\section{Proof of limit formula}\label{limit-sec}

In this section we prove Lemma \ref{equidist}.  Let $p, G, \mathrm{X}, k, c_0,\ldots,c_k, (\Phi_n)_{n=1}^\infty$ be as in that lemma.  
If $(\pi_{k-1})_* f_i = 0$ for some $i=0,\ldots,k$ then the claim is immediate from Theorem \ref{char-thm}.  By linearity, we may thus reduce to the case when each $f_i$ is a pullback by $(\pi_{k-1})^*$ from the associated function $\tilde f_i := (\pi_{k-1})_* f_i$.
Our task is to show that for any $f_0,\ldots,f_k \in L^\infty(\mathrm{X})$, the expression
\begin{equation}\label{cake}
\E_{g \in \Phi_n} \int_X f_0(T_{c_0 g} x) \ldots f_k(T_{c_kg} x)\ d\mu(x)
\end{equation}
converges as $n \to \infty$ to the integral
\begin{equation}\label{sold}
 \int_{HP_{c_0,\ldots,c_k}(\mathrm{Z}_{k-1})} \tilde f_0 \otimes\ldots \otimes \tilde f_k \ dm_{HP_{c_0,\ldots,c_k}(\mathrm{Z}_{k-1})}.
\end{equation}

As noted after the statement of Lemma \ref{equidist}, $HP_{c_0,\ldots,c_k}(\mathrm{Z}_{k-1})$ surjects onto each of the $k+1$ coordinates of $(Z_{k-1})^{k+1}$.  As such, we can bound \eqref{sold} in magnitude by $\|f_i\|_{L^2(\mathrm{X})}$ for any $0 \leq i \leq k$, if we normalize so that $\|f_j\|_{L^\infty(\mathrm{X})} \leq 1$ for $j \neq i$.  Of course, a similar bound also can be obtained for \eqref{cake}.  By combining these observations with Fourier analysis on the compact abelian group $U_1 \times \ldots \times U_k$ and a limiting\footnote{Here we use the basic fact that an $L^\infty$ function on a compact abelian group can be approximated to arbitrary accuracy in $L^2$ norm by a finite linear combination of multiplicative characters, while still staying uniformly bounded in $L^\infty$.  This can be established for instance by first approximating the function by a continuous function, then using the Stone-Weierstrass theorem.} argument, it suffices to verify this claim under the assumption that each $\tilde f_i, i=0,\ldots,k$ is a tensor product of multiplicative characters, thus
$$ \tilde f_i( u_1,\ldots,u_{k-1}) = \prod_{j=1}^{k-1} e(\phi_{ij}(u_j))$$
for all $u_1 \in U_1,\ldots,u_{k-1} \in U_{k-1}$ and some additive characters (i.e continuous homomorphisms) $\phi_{ij}: U_j \to \R/\Z$ for $i=0,\ldots,k$ and $j=1,\ldots,k-1$.  The expression \eqref{cake} is then equal to
\begin{equation}\label{phi1}
\E_{g \in \Phi_n} \int_X e\left( \sum_{i=0}^k \sum_{j=1}^{k-1} \phi_{ij}\left( \psi_{j,g,x}(c_i) \right) \right)\ d\mu(x) 
\end{equation}
where $\psi_{j,g,x}: \Z \to U_j$ is the (periodic) sequence
\begin{equation}\label{psijg}
\psi_{j,g,x}(n) := u_j\left( \pi_{k-1}( T_{n g} x ) \right)
\end{equation}
and $u_j: Z_{k-1} \to U_j$, $j=1,\ldots,k-1$ are the coordinate functions.  Also, by Fourier analysis, the expression \eqref{sold} is equal to $1$ when we have the identities
\begin{equation}\label{phi2}
 \sum_{i=0}^k \phi_{ij}\left( P_j(c_i) \right) = 0
\end{equation}
for all $j=1,\ldots,k-1$ and all polynomials $P_j: \Z \to U_j$ of degree $<j+1$, and zero otherwise.

The strategy is to use the polynomial structure of the Weyl system to place the additive characters $\phi_{ij}$ in a ``normal form'', at which point the convergence can be deduced from Lemma \ref{vdc}.  This is achieved as follows.
From construction of the Weyl system we have
$$ \Delta_g u_j = \rho_j(g,\cdot)$$
for all $g \in G$ and $j=1,\ldots,k-1$.  Since $\rho_j$ is a polynomial cocycle of degree $<j$, we conclude that $u_j: Z_{k-1} \to U_j$ is a polynomial of degree $<j+1$.  This implies that for any $x \in X$, the sequence $\psi_{j,g,x}$ defined in \eqref{psijg} is a polynomial sequence of degree $<j+1$, and thus has a Taylor expansion of the form
$$ u_j( \pi_{k-1}( T_{ng} x ) ) = \sum_{l=0}^j \binom{n}{l} a_{j,g,x}$$
for some coefficients $a_{j,g,x} \in U_j$.  As the $c_0,\ldots,c_j$ are distinct elements of $\F_p$, we may then use Lagrange interpolation (using the $p$-torsion nature of $U_j$ and the hypothesis $j<p$ to justify any division occuring in the interpolation formula), this implies that one can express $\psi_{j,g,x}(n)$ as a linear combination of $\psi_{j,g,x}(c_0),\ldots,\psi_{j,g,x}(c_j)$, thus
\begin{equation}\label{uj}
 \psi_{j,g,x}(n) = \sum_{i=0}^j b_{n,j,i} \psi_{j,g,x}(c_i)
\end{equation}
for some integer coefficients $b_{n,j,i}$ that do not depend on $g$ or $x$ (but may depend on $p$ and $c_0,\ldots,c_j$).  Indeed, the interpolation formula gives the more general identity
$$ P_j(n) = \sum_{i=0}^j b_{n,j,i} P_j(c_i)$$
for any polynomial $P_j: \Z \to U_j$ of degree $<j+1$, with the same coefficients $b_{n,j,i}$.  In particular, for any $j < i \leq k$, we can rewrite $\phi_{ij}( \psi_{j,g,x}(c_i) )$ in \eqref{phi1} as a linear combination of the $\phi_{ij}( \psi_{j,g,x}(c_0) ), \ldots, \phi_{ij}( \psi_{j,g,x}(c_j) )$, and similarly write $\phi_{ij}(P_j(c_i))$ in \eqref{phi2} as the same linear combination of the $\phi_{ij}(P_j(c_0)),\ldots,\phi_{ij}(P_j(c_j))$.  From this fact, we see that if the additive character $\phi_{ij}$ is not identically zero for some $j < i\leq k$, we may delete that character (and adjust the characters $\phi_{0j},\ldots,\phi_{jj}$ by appropriate multiples of the deleted character) without affecting either \eqref{phi1} or \eqref{phi2}.
Using this observation repeatedly, we see that to prove the convergence of \eqref{phi1} to $1$ when \eqref{phi2} holds and zero otherwise, it suffices to do so under the normalization that $\phi_{ij}=0$ for all $i>j$, which we now assume henceforth.

We now divide the argument into two cases.  If the $\phi_{ij}$ are all identically zero, then the claim is trivial.  Otherwise, we may find $1 \leq j_* \leq k$ such that $\phi_{ij}$ all vanish for $j > j_*$, but $\phi_{i_* j_*}$ is not identically zero for at least one $0 \leq i_* \leq j_*$.  By permuting the $i$ indices, and then readjusting the $\phi_{ij}$ characters for $j < j_*$ as before, we may assume without loss of generality that $i_*=j_*$.

Observe from Lagrange interpolation that if $P_{j_*}: \Z \to U_{j_*}$ is an arbitrary polynomial sequence of degree $<j_*+1$, then the tuple $(P_{j_*}(c_0),\ldots,P_{j_*}(c_{j_*}))$ can take arbitrary values in $U_{j_*}^{j_*+1}$; in particular, as $\phi_{j_* j_*}$ is not identically zero, the identity \eqref{phi2} does \emph{not} hold for $j=j_*$.  Thus, the expression \eqref{sold} vanishes in this case, and our task is now to show that \eqref{phi1} converges to zero.  But from the vanishing of $\phi_{ij}$ when $i > j$ or $j > j_*$, we can write \eqref{phi1} in the form
$$
\E_{g \in \Phi_n} \int_X \left(\prod_{i=0}^{j_*-1} F_i(T_{c_i g} x)\right)  e\left( \phi_{j_* j_*}\left( u_{j_*}( \pi_{k-1}(T_{j_* g} x) ) \right) \right)\ d\mu(x)$$
for some functions $F_0,\ldots,F_{j_*-1} \in L^\infty(\mathrm{X})$ of unit magnitude which do not depend on $g$ or $x$, and whose exact form will not be important to us.  Applying Lemma \ref{vdc}, we conclude that
$$
\limsup_{n \to \infty}
\left|\E_{g \in \Phi_n} \int_X e\left( \sum_{i=0}^k \sum_{j=1}^m \phi_{ij}\left( \psi_{j,g,x}(c_i) \right) \right)\ d\mu(x)\right| \leq \| e( \phi_{j_* j_*}(u_{j_*}(\pi_{k-1})) ) \|_{U^{j_*}(\mathrm{X})}.$$
However, as the character $\phi_{j_* j_*}$ is not identically zero, the function $e( \phi_{j_* j_*}(u_{j_*})$ has mean zero on every coset of $U_{j_*}$ in $U_1 \times \ldots \times U_{k-1}$, and thus
$$ (\pi_{j_*-1})_* \left( e\left( \phi_{j_* j_*}\left(u_{j_*}(\pi_{k-1})\right) \right) \right) = 0.$$
By Theorem \ref{char-uk}, we conclude that
$$ \left\| e\left( \phi_{j_* j_*}\left(u_{j_*}(\pi_{k-1})\right) \right) \right\|_{U^{j_*}(\mathrm{X})} = 0$$
giving the desired convergence of \eqref{phi1} to zero.  This concludes the proof of Lemma \ref{equidist}.

\begin{remark}\label{norm}  The above argument gives a new proof of the convergence of the averages
$$ \E_{g \in \Phi_n} \int_X (T_{c_0 g} f_0) \ldots (T_{c_kg} f_k)\ d\mu$$
as $n \to \infty$ for arbitrary $k \geq 0$, $f_0,\ldots,f_k \in L^\infty(\mathrm{X})$, and $c_0,\ldots,c_k \in \F_p$ (since, after collecting like terms, we can reduce to the case where the $c_0,\ldots,c_k \in \F_p$ are distinct, so that $k < p$).   A modification of the argument also shows convergence in $L^2(\mathrm{X})$ of the averages
\begin{equation}\label{goof}
 \E_{g \in \Phi_n} (T_{c_1 g} f_1) \ldots (T_{c_kg} f_k)
 \end{equation}
for arbitrary $k \geq 0$, $c_1,\ldots,c_k \in \F_p$ and $f_1,\ldots,f_k \in L^\infty(\mathrm{X})$.  We sketch the argument as follows.  Firstly, by collecting like terms and factoring out those terms with $c_i=0$, we may assume that the $c_1,\ldots,c_k$ are distinct and non-zero, so that $k<p$.  By Theorem \ref{char-thm} (as in the proof of Lemma \ref{equidist}), we may assume that each $f_i$ is of the form $f_i = \pi_{k-1}^* \tilde f_i$ for some $\tilde f_i \in L^\infty(\mathrm{Z}_{k-1})$, and then we can use Fourier decomposition as before to assume that each $\tilde f_i$ is the tensor product of characters $e(\phi_{ij})$.  We can then use identities of the form \eqref{uj} (setting $c_0 := 0$) to reduce to the case where the $\phi_{ij}$ vanish for $i>j$, and then one can adapt the preceding argument to show that the average \eqref{goof} either is identically $1$, or converges in norm to zero. We leave the details to the interested reader.
We also remark that the limit value of \eqref{goof} does not depend on the F\o lner sequence $(\Phi_n)$.
\end{remark}

\section{Proof of second limit formula}\label{second-limit}

We now give a proof of Lemma \ref{equidist-2}.  This will be a minor variant of the argument used to prove Lemma \ref{equidist}.

Let $p, G, \mathrm{X}, k, c_0,\ldots,c_k, (\Phi_n)_{n=1}^\infty$ be as in that lemma.  Using Theorem \ref{char-thm} as in the previous section (but with $\mathrm{Z}_k$ as the characteristic factor, instead of $\mathrm{Z}_{k-1}$), we may reduce to the case when each $f_i$ is a pullback by $(\pi_k)^*$ from the associated function $\tilde f_i := (\pi_k)_* f_i$.
Our task is to show that for any $f_0,\ldots,f_k \in L^\infty(\mathrm{X})$, the expression
\begin{equation}\label{cake-2}
|I_{c_0,\ldots,c_k;f_0,\ldots,f_k}(g) - J_{c_0,\ldots,c_k;f_0,\ldots,f_k}(g)|
\end{equation}
converges in uniform density to zero.

Observe that the closed group $HP_{c_0,\ldots,c_k}(\mathrm{Z}_{k})_0$ contains the diagonal group $Z_k^\Delta := \{ (z,\ldots,z): z \in Z_k \} \subset Z_k^{k+1}$, and thus surjects onto each factor $Z_k$.  Translating, we see that the cosets $HP_{c_0,\ldots,c_k}(\mathrm{Z}_{k})_g$ also surject onto each factor $Z_k$. We can then repeat the limiting argument from the previous section and reduce to the case that each $\tilde f_i, i=0,\ldots,k$ is a tensor product of characters, thus
$$ \tilde f_i( u_1,\ldots,u_k) = \prod_{j=1}^k e(\phi_{ij}(u_j))$$
for all $u_1 \in U_1,\ldots,u_k \in U_k$ and some characters (i.e continuous homomorphisms) $\phi_{ij}: U_j \to \R/\Z$ for $i=0,\ldots,k$ and $j=1,\ldots,k$.  For any $g \in G$, the expression $I_{c_0,\ldots,c_k;f_0,\ldots,f_k}(g)$ is then equal to
\begin{equation}\label{phi1-2}
\int_X e( \sum_{i=0}^k \sum_{j=1}^{k} \phi_{ij}( \psi_{j,g,x}(c_i) ) )\ d\mu(x) 
\end{equation}
where
\begin{equation}\label{psij-2}
\psi_{j,g,x}(n) := u_j( \pi_{k}( T_{n g} x ) )
\end{equation}
and $u_j: Z_{k} \to U_j$, $j=1,\ldots,k$ are the coordinate functions.  Meanwhile, the value of $J_{c_0,\ldots,c_k;f_0,\ldots,f_k}(g)$ depends on the behavior of the quantities
\begin{equation}\label{phi2-2}
\sum_{i=0}^k \phi_{ij}( P_j(c_i) ) 
\end{equation}
for $j=1,\ldots,k$, where $P_j$ ranges over all polynomials $P_j: \Z \to U_j$ of degree $<j+1$ with leading coefficient $\Delta_1^j P_j = \Delta_g^j u_j$.  If, for each $j$, the expression \eqref{phi2-2} is equal to a constant $\theta_{j,g} \in \R/\Z$ independent of $P_j$, then the expression $J_{c_0,\ldots,c_k;f_1,\ldots,f_k}(g)$ is equal to $e(\sum_{j=1}^k\theta_{j,g})$.  In all other cases, $J_{c_0,\ldots,c_k;f_1,\ldots,f_k}(g)$ vanishes.

As in the previous section, we use the polynomial structure of the Weyl system to place the characters $\phi_{ij}$ in a ``normal form''.  Fix $g \in G$.  As before, for each $1 \leq j \leq k$ and $x \in \mathrm{X}$, the sequence $\psi_{j,g,x}: \Z \to U_j$ is a polynomial sequence of degree $<j+1$ from $\Z$ to $U_j$.  However, because $g$ is now fixed, we see from \eqref{psij-2} that we have an additional constraint on the top order coefficient of $\psi_{j,g,x}$:
\begin{equation}\label{phin}
 \Delta_1^j \psi_{j,g,x} = \Delta_g^j u_j.
\end{equation}
This additional ($g$-dependent) constraint on $\phi_{j,g,x}$ will allow us to eliminate one further character $\phi_{ij}$ than was possible in the previous section.  Indeed, from \eqref{phin} we see that the modified sequence
$$ n \mapsto \psi_{j,g,x}(n) - \binom{n}{j} \Delta_g^j u_j$$
is now a polynomial sequence of degree $<j$ rather than $<j+1$.  Applying Lagrange interpolation to this polynomial of one lower degree and then rewriting everything in terms of $\psi_{j,g,x}$, one obtains identities of the form
$$
 \psi_{j,g,x}(n) = \sum_{i=0}^{j-1} b'_{n,j,i,g} \psi_{j,g,x}(c_i) + a'_{n,j,i}$$
for all $n \in \Z$ and some coefficients $b'_{n,j,i,g}, a'_{n,j,g} \in \F_p$ that do not depend on $x$.  Furthermore, we have the same identity
$$
 P_j(n) = \sum_{i=0}^{j-1} b'_{n,j,i,g} P_j(c_i) + a'_{n,j,i}$$
for any polynomial $P_j: \Z \to U_j$ of degree $< j+1$ obeying the constraint $\Delta_1^j P_j = \Delta_g^j u_j$.

Because of these identities, we see that if $\phi_{ij}$ is not identically zero for some $1 \leq j \leq i \leq k$, then one can rewrite $\phi_{ij}( \psi_{j,g,x}(c_i) )$ as a linear combination of $\phi_{ij}( \psi_{j,g,x}(c_0) ), \ldots, \phi_{ij}( \psi_{j,g,x}(c_{j-1}) )$ plus a constant independent of $x$, and similarly one can rewrite the expression $\phi_{ij}(P_j(c_i))$ in \eqref{phi2-2} as the same linear combination of $\phi_{ij}(P_j(c_0)),\ldots,\phi_{ij}(P_j(c_{j-1}))$ plus the same constant.  Because of this, we can delete $\phi_{ij}$ (and adjust the characters $\phi_{0j},\ldots,\phi_{j-1,j}$ by appropriate multiples of the deleted character), resulting in $I_{c_0,\ldots,c_k;f_1,\ldots,f_k}(g)$ and
$J_{c_0,\ldots,c_k;f_1,\ldots,f_k}(g)$ being rotated by the same ($g$-dependent) phase shift.  In particular, the expression \eqref{cake-2} remains unchanged by these modifications of the characters $\phi_{ij}$.  By arguing as in the previous section, we may thus reduce to the case when the $\phi_{ij}$ vanish for all $j \leq i \leq k$ (note carefully that this is a slightly stronger vanishing criterion than in the previous section, when we only had $\phi_{ij}$ vanish for $j < i \leq k$).  

As in the preceding section, we now divide into two cases.  If the $\phi_{ij}$ are all identically zero, then the claim is trivial.  Otherwise, we may find $1 \leq j_* \leq k$ such that $\phi_{ij}$ all vanish for $j > j_*$, but $\phi_{i_* j_*}$ is not identically zero for at least one $0 \leq i_* \leq j_*-1$.  By permuting the $i$ indices, and then readjusting the $\phi_{ij}$ characters for $j < j_*$ as before, we may assume without loss of generality that $i_*=j_*-1$.

Fix $g \in G$.  Observe from Lagrange interpolation that if $P_{j_*}: \Z \to U_{j_*}$ is a polynomial sequence of degree $<j_*+1$ that is arbitrary save for obeying the constraint $\Delta_1^{j_*} P_{j_*} = \Delta_g^{j_*} u_{j_*}$, then the sequence $n \mapsto P_{j_*}(n) - \binom{n}{j_*} \Delta_g^{j_*} u_{j_*}$ is an arbitrary polynomial of degree $<j_*$.  In particular, the tuple $(P_{j_*}(c_0),\ldots,P_{j_*}(c_{j_*-1}))$ can take arbitrary values in $U_{j_*}^{j_*}$.  Thus, as $\phi_{j_*-1, j_*}$ is not identically zero, the identity \eqref{phi2} does \emph{not} hold for $j=j_*$, and so $J_{c_0,\ldots,c_k;f_0,\ldots,f_k}(g)$ vanishes for all $g \in G$.  Our task is now to show that $I_{c_0,\ldots,c_k;f_0,\ldots,f_k}(g)$ converges in uniform density to zero.  But from the vanishing of $\phi_{ij}$ when $i \geq j$ or $j > j_*$, we can write \eqref{phi1} in the form
$$
\E_{g \in \Phi_n} \int_X \left(\prod_{i=0}^{j_*-2} F_i(T_{c_i g} x)\right)  e\left( \phi_{j_*-1, j_*}\left( u_{j_*}\left( \pi_{k-1}( T_{c_{j_*-1} g} x) \right) \right) \right)\ d\mu(x)$$
for some functions $F_0,\ldots,F_{j_*-2} \in L^\infty(\mathrm{X})$ of unit magnitude which do not depend on $g$ or $x$.  Arguing as in the previous section, we have
$$ \left\|  e\left( \phi_{j_*-1, j_*}\left( u_{j_*}( \pi_{k-1} ) \right) \right) \right\|_{U^{j_*}(\mathrm{X})} = 0$$
and the claim now follows from Lemma \ref{vdc}.

\section{Proof of structure theorem}\label{struct-sec}

We now prove Proposition \ref{struct}.  Let the notation be as in Lemma \ref{equidist-2}.  Observe that the coset $HP_{c_0,\ldots,c_k}(\mathrm{Z}_{k})_g$ only depends on $g$ through the quantities $\Delta_g^j u_j \in U_j$ for $j=1,\ldots,k$.  Furthermore, the dependence of the integral
$$ J_{c_0,\ldots,c_k;f_0,\ldots,f_k}(g) =
\int_{HP_{c_0,\ldots,c_k}(\mathrm{Z}_{k-1})} (\pi_k)_* f_0 \otimes\ldots \otimes (\pi_k)_* f_k \ dm_{HP_{c_0,\ldots,c_k}(\mathrm{Z}_{k})_{\theta(g)}}.$$
on $\theta(g)$ is continuous (this is easiest to see by first approximating each $(\pi_i)_* f_i$ in $L^2$ norm with a continuous function on the compact group $Z_k$).  From this, we see that it suffices to represent each sequence $\theta_j: g \mapsto \Delta_g^j u_j$ for $j=1,\ldots,k$ in the form
$$ \theta_j(g)= \Delta_g^j u_j = F_j( S_{j,g} y_j )$$
for some continuous $j$-step Weyl system $\mathrm{Y}_j = (Y_j, \Y_j, \nu_j, (S_{j,g})_{g \in G})$, some continuous function $F_j: Y_j \to U_j$, and some point $y_j \in Y$.  As the claim then follows by composing the various continuous maps together (and noting that the product of finitely many continuous Weyl systems of step at most $k$ will be of step at most $k$).

For each $j$, we introduce the form $\Lambda_j: G^j \to U_j$ by the formula
$$ \Lambda_j( g_1,\ldots,g_j) := \Delta_{g_1} \ldots \Delta_{g_j} u_j$$
(again, note that the right-hand side is a constant function and so can be identified with an element of $U_j$).  From the identities $\Delta_g \Delta_h = \Delta_h \Delta_g$ and $\Delta_{g+h} = \Delta_g + \Delta_h + \Delta_g \Delta_h$ (and noting that any $j+1$-fold derivative of $u_j$ vanishes) we see that $\Lambda_j$ is a symmetric multilinear form.  Our task is to establish a representation of the form
\begin{equation}\label{g-rep}
 \Lambda_j(g,\ldots,g) = F_j( S_{j,g} y_j )
\end{equation}
for some continuous $k$-step Weyl system $\mathrm{Y}_j = (Y_j, \Y_j, \nu_j, (S_{j,g})_{g \in G})$, some continuous function $F_j: Y_j \to U_j$, and some point $y_j \in Y$.  

Fix $j$.  To achieve the above goal, we will exploit a dynamical abstraction of the algebraic observation (essentially the binomial formula) that if we define
$$ \Gamma_i(x)(h_1,\ldots,h_{j-i}) := \Lambda_j(h_1,\ldots,h_{j-i}, x^{(i)})$$
for $0 \leq i \leq j$ and $x,h_1,\ldots,h_{j-i} \in G$, where $x^{(i)}$ denotes $i$ copies of $x$, (so in particular $\Gamma_0(x) = \Lambda_j$), then the $\Gamma_i(x): G^{j-i} \to U_j$ are symmetric multilinear forms (where the multilinearity is of course with respect to the field $\F_p$) and we have the shift identity
$$ \Gamma_i(x+g)(h_1,\ldots,h_{j-i}) = \sum_{l=0}^i \binom{l}{i} \Gamma_l(x)( h_1,\ldots,h_{j-i}, g^{(i-l)} )$$
for all $0 \leq i \leq j$ and $x,g,h_1,\ldots,h_{j-i} \in G$.

Now we give the dynamical version of the above identity.  For each $1 \leq i\leq j$, let $V_i$ be the collection of all symmetric multilinear forms $\Gamma_i: G^{j-i} \to U_j$, where the multilinearity is of course with respect to the field $\F_p$.  This space $V_i$ can be viewed as a closed subgroup of the product space $U_i^{G^{j-i}}$ and is thus a compact abelian group.  Set $Y_j := V_1 \times \ldots \times V_j$ with the product $\sigma$-algebra $\Y_j$ and Haar measure $\nu_j$.  We define the shift maps
$$ S_{j,g} ( \Gamma_1,\ldots,\Gamma_j ) = ( S_{j,g} ( \Gamma_1,\ldots,\Gamma_j )_i)_{i=1}^j$$
for $g \in G$ and $\Gamma_i \in V_i$ for $i=1,\ldots,j$ by the formula
\begin{equation}\label{fm}
 S_{j,g} ( \Gamma_1,\ldots,\Gamma_j )_i(h_1,\ldots,h_{j-i}) = \sum_{l=0}^i \binom{i}{l} \Gamma_l( h_1,\ldots,h_{j-i}, g^{(i-l)} )
\end{equation}
with the convention that $\Gamma_0 := \Lambda_j$. We verify that this is an action:
\begin{align*}
S_{j,g'}(S_{j,g} ( \Gamma_1,\ldots,\Gamma_j ))_i(h_1,\ldots,h_{j-i})
&= 
\sum_{l=0}^i \binom{i}{l} S_{j,g} ( \Gamma_1,\ldots,\Gamma_j )_l( h_1,\ldots,h_{j-i}, (g')^{(i-l)} ) \\
&= \sum_{l=0}^i \sum_{m=0}^l \binom{i}{l} \binom{l}{m}
\Gamma_m( h_1,\ldots,h_{j-i}, (g')^{(i-l)}, g^{(l-m)} )  \\
&= \sum_{m=0}^i \binom{i}{m} \sum_{a=0}^{i-m} \binom{i-m}{a} 
\Gamma_m( h_1,\ldots,h_{j-i}, (g')^{(a)}, g^{(i-m-a)} )  \\
&= \sum_{m=0}^i \binom{i}{m} \Gamma_m( h_1,\ldots,h_{j-i}, (g+g')^{(i-m)} )  \\
&= S_{j,g+g'}( \Gamma_1,\ldots,\Gamma_j )_i(h_1,\ldots,h_{j-i})
\end{align*}
where the penultimate equation follows from the symmetry and multilinearity of $\Gamma_m$, and we have implicitly used the fact that the formula \eqref{fm} extends to the $i=0$ case with the convention that $ S_{j,g} ( \Gamma_1,\ldots,\Gamma_j )_0$ and $\Gamma_0$ are both equal to $\Lambda_j$.  

One easily verifies by induction that $\mathrm{Y}_j = (Y_j, \Y_j, \nu_j, (S_{j,g})_{g \in G})$ is a tower
$$ \mathrm{Y}_j = V_1 \times_{\eta_2} V_2 \times_{\eta_3} \ldots \times_{\eta_j} V_j$$
of cocycle extensions, where the $(G,V_1 \times_{\eta_2} \ldots \times_{\eta_{i-1}} V_{i-1}, V_i)$-cocycle $\eta_i$ is defined for $1 \leq i \leq j$ by the formula
\begin{equation}\label{ling}
\eta_i( g, ( \Gamma_1,\ldots,\Gamma_{i-1} ) )(h_1,\ldots,h_{j-i}) = \sum_{l=0}^{i-1} \binom{i}{l} \Gamma_l( h_1,\ldots,h_{j-i}, g^{(i-l)} ).
\end{equation}
For each $1 \leq i \leq j$, we see from \eqref{fm} that whenever one differentiates the coordinate function $v_i: (\Gamma_1,\ldots,\Gamma_l) \mapsto \Gamma_i$ in some direction $j$, one obtains an affine-linear combination of the previous coordinate functions $v_1,\ldots,v_{i-1}$.  In particular, this implies that each coordinate function $v_i$ is a polynomial of degree $<i+1$, which implies from \eqref{ling} that each cocycle $\eta_i$ is a polynomial of degree $<i$.  Thus $\mathrm{Y}_j$ is a $j$-step Weyl system; an easy induction then shows that it is in fact a continuous $j$-step Weyl system.

From \eqref{fm} we have
$$ v_j( S_{j,g}( 0,\ldots,0 ) ) = \Lambda_j(g,\ldots,g);$$
as $v_j: Y_j \to V_j \equiv U_j$ is clearly a continuous function, we obtain the desired representation \eqref{g-rep}.

\section{Khintchine for double recurrence}\label{double-sec}

We now prove Theorem \ref{double}.  Suppose for contradiction that the claim failed.  Then we could find $p, \mathrm{X}, A, \eps$ as in Definition \ref{khin-def} such that the set
$$ \{ \mu( T_{-c_0 g} A \cap T_{-c_1g} A \cap T_{-c_2 g} A ) \geq \mu(A)^3-\eps \}$$
failed to be syndetic.  In particular, the complement of this set contains translates of any given finite set, and in particular must contain a F{\o}lner sequence $(\Phi_n)_{n=1}^\infty$, thus we have
$$ \mu( T_{-c_0 g} A \cap T_{-c_1g} A \cap T_{-c_2 g} A ) < \mu(A)^3-\eps $$
for all $n=1,2,\ldots$ and $g \in \Phi_n$.  We can rewrite this as
\begin{equation}\label{sale}
 \int_X (T_{c_0 g} 1_A) (T_{c_1 g} 1_A) (T_{c_2 g} 1_A)\ d\mu < \mu(A)^3-\eps,
\end{equation}
where $1_A$ is of course the indicator function of $A$.
If we could show that
$$ \lim_{n \to \infty} \E_{g \in \Phi_n}\int_X (T_{c_0 g} 1_A) (T_{c_1 g} 1_A) (T_{c_2 g} 1_A)\ d\mu  > \mu(A)^3-\eps$$
then we would obtain the desired contradiction.  Unfortunately, a direct application of Lemma \ref{equidist} computes the left-hand side as
$$ \int_{U_1} \int_{U_1} f(x+c_0 t) f(x+c_1 t) f(x+c_2 t)\ dm_{U_1}(x) dm_{U_1}(t),$$
where $U_1 = Z_1$ is the Kronecker factor and $f := \pi_1 1_A$, and it is possible\footnote{An explicit example of this phenomenon can be constructed, for large $p$ at least, by adapting the Behrend construction \cite{behrend}, similar to the construction in \cite[Section 2.1]{bhk}, which handled the case $c_i=i$ in which $U_1$ was replaced by $\R/\Z$ and $f$ replaced by an indicator function $1_B$; we omit the details.} for this integral to be significantly smaller than $\mu(A)^3 = (\int_{U_1} f\ dm_{U_1})^3$.  However, we can get around this difficulty by the following trick of Frantzikinakis \cite{nikos} (see also \cite{bhk}).  Observe from H\"older's inequality that
$$
\int_{U_1} f(x) f(x) f(x)\ dm_{U_1}(x) \geq (\int_{U_1} f\ dm_{U_1})^3 = \mu(A)^3.$$
As translations are continuous in (say) the $L^2$ norm, we conclude that
$$
\int_{U_1} f(x+c_0 t) f(x+c_1 t) f(x+c_2 t)\ dm_{U_1}(x) \geq \mu(A)^3 - \eps/2$$
(say) for all $t \in U_1$ sufficiently close to the origin.  In particular, by Urysohn's lemma, we can find a nonnegative continuous function $\eta: U_1 \to \R^+$ with $\int_{U_1} \eta\ dm_{U_1} = 1$ such that
\begin{equation}\label{sale-2}
\int_{U_1} \int_{U_1} \eta(t) f(x+c_0 t) f(x+c_1 t) f(x+c_2 t)\ dm_{U_1}(x) dm_{U_1}(t) \geq \mu(A)^3 - \eps/2.
\end{equation}
We now claim the weighted limit formula
\begin{equation}\label{modlim}
\begin{aligned}
&\lim_{n \to \infty} \E_{g \in \Phi_n} 
\eta(\rho_1(g)) \int_X (T_{c_0 g} 1_A) (T_{c_1 g} 1_A) (T_{c_2 g} 1_A)\ d\mu   \\
&= \int_{U_1} \int_{U_1} \eta(t) f(x+c_0 t) f(x+c_1 t) f(x+c_2 t)\ dm_{U_1}(x) dm_{U_1}(t)   
\end{aligned}\end{equation}
which (on replacing $A$ by all of $X$) gives
$$\lim_{n \to \infty} \E_{g \in \Phi_n} \eta(\rho_1(g)) = 1$$
(this can also be established from the unique ergodicity of the Kronecker factor), and this will gives a contradiction between \eqref{sale} and \eqref{sale-2}.

It remains to establish \eqref{modlim}.  We will in fact show this formula for arbitrary continuous functions $\eta: U_1 \to \C$.  By the Stone-Weierstrass theorem and Fourier analysis, it suffices to establish the claim when $\eta$ is a multiplicative character, thus $\eta = e(\phi)$ for some continuous homomorphism $\phi: U_1 \to \R/\Z$.  But as $c_0,c_1,c_2$ are distinct elements of $\F_p$, there is a Lagrange interpolation identity of the form
$$ t = a_0 (x+c_0 t) + a_1 (x+c_1 t) + a_2 (x+c_2 t)$$
for some integers $a_0,a_1,a_2$ depending only on $c_0,c_1,c_2$.  As such, the right-hand side of \eqref{modlim} can be rewritten as
$$ \int_{U_1} \int_{U_1} f_0(x+c_0 t) f_1(x+c_1 t) f_2(x+c_2 t)\ dm_{U_1}(x) dm_{U_1}(t)$$
where $f_i(x) := f(x) e( a_i \phi(x) )$ for $i=0,1,2$.  Meanwhile, as the shift $(T_{1,g})_{g \in G}$ on the Kronecker system $Z_1 = U_1$ is given by $T_{1,g}: x \mapsto x + \rho_1(g)$ for each group element $g \in G$, we can write
$$ \rho_1(g) = a_0 T_{1,c_0 g}(x) + a_1 T_{1,c_1 g}(x) + a_2 T_{1,c_2 g}(x)
$$
for all $g \in G$ and $x \in U_1$, which implies that
$$ \eta(\rho_1(g)) \int_X (T_{c_0 g} 1_A) (T_{c_1 g} 1_A) (T_{c_2 g} 1_A)\ d\mu = \int_X (T_{c_0 g} \tilde f_0) (T_{c_1 g} \tilde f_1) (T_{c_2 g} \tilde f_2)\ d\mu$$
for any $g \in G$, where $\tilde f_i(x) := 1_A(x) e(a_i \phi(\pi_1(x)))$ for $i=0,1,2$ and $x \in X$.  Since $f_i = (\pi_1)_* \tilde f_i$ for $i=0,1,2$, the claim \eqref{modlim} now follows from Lemma \ref{equidist}.

\section{Khintchine for triple recurrence}\label{triple-sec}

We now give the proof of Theorem \ref{triple}, which follows the lines of Theorem \ref{double} (and is of course also similar to the proof of the analogous claim for $\Z$-systems in \cite{bhk}).  We may permute indices so that $c_0+c_3=c_1+c_2$.  By translating we may normalize $c_0=0$, and then by dilating we may normalize $c_1=1$, so that $c_3 = c_2+1$.

Again, we assume for contradiction that the claim failed, then by arguing as before we can find $p, \mathrm{X}, A, \eps$ verifying the hypotheses of that theorem and a F{\o}lner sequence $(\Phi_n)_{n=1}^\infty$ such that
$$ \int_X 1_A (T_{g} 1_A) (T_{c_2 g} 1_A) (T_{(c_2+1) g} 1_A)\ d\mu < \mu(A)^4 - \eps$$
for all $n =1,2,\ldots$ and $g \in \Phi_n$.  Arguing as before, we see that it suffices to locate a non-negative continuous function $\eta: U_1 \to \R^+$ with $\int_{U_1} \eta\ dm_{U_1} = 1$ such that
$$ \lim_{n \to \infty} \E_{g \in \Phi_n} \eta(\rho_1(g)) \int_X 1_A (T_{g} 1_A) (T_{c_2 g} 1_A) (T_{(c_2+1) g} 1_A)\ d\mu \geq \mu(A)^4 - \eps/2.$$

A direct application of Lemma \ref{equidist} gives
\begin{align*}
&\lim_{n \to \infty} \E_{g \in \Phi_n} \int_X 1_A (T_{g} 1_A) (T_{c_2 g} 1_A) (T_{(c_2+1) g} 1_A)\ d\mu \\
&= 
\int_{U_1^2 \times U_2^3} f(x_1, x_2) f(x_1+t_1, x_2+t_2)
f\left(x_1+c_2 t_1, x_2+c_2 t_2+ \binom{c_2}{2} u_2\right)\\
 &\qquad  \qquad \
f\left(x_1+(c_2+1) t_1, x_2+(c_2+1) t_2+\binom{c_2+1}{2} u_2\right) \\
&\qquad \qquad \  dm_{U_1}(x_1) dm_{U_1}(t_1) dm_{U_2}(x_2) dm_{U_2}(t_2) dm_{U_2}(u_2),
\end{align*}
where we now set $f := (\pi_2)_* 1_A$.  We can twist this identity by characters as in the previous section to conclude the weighted generalization
\begin{align*}
&\lim_{n \to \infty} \E_{g \in \Phi_n} \eta(\rho_1(g)) \int_X 1_A (T_{g} 1_A) (T_{c_2 g} 1_A) (T_{(c_2+1) g} 1_A)\ d\mu  \\
&= 
\int_{U_1^2 \times U_2^3} \eta(t_1)\quad f(x_1, x_2) f(x_1+t_1, x_2+t_2)
f\left(x_1+c_2 t_1, x_2+c_2 t_2+ \binom{c_2}{2} u_2\right) \\
 & \qquad \qquad \ 
f\left(x_1+(c_2+1) t_1, x_2+(c_2+1) t_2+\binom{c_2+1}{2} u_2\right) \\
&\qquad \qquad \  dm_{U_1}(x_1) dm_{U_1}(t_1) dm_{U_2}(x_2) dm_{U_2}(t_2) dm_{U_2}(u_2).
\end{align*}
By Urysohn's lemma and the Fubini-Tonelli theorem, it thus suffices to show that
\begin{align*}
& \int_{U_1 \times U_2^3} \quad f(x_1, x_2) f(x_1+t_1, x_2+t_2)
f\left(x_1+c_2t_1, x_2+c_2t_2+\binom{c_2}{2} u_2\right)\\
& \qquad \qquad \ f\left(x_1+(c_2+1)t_1, x_2+(c_2+1)t_2+\binom{c_2+1}{2} u_2\right) \\
&\qquad \qquad \ dm_{U_1}(x_1) dm_{U_2}(x_2) dm_{U_2}(t_2) dm_{U_2}(u_2) \\
&\quad \geq \mu(A)^4 - \eps/2
\end{align*}
for all $t_1$ sufficiently close to the origin.   As before, this expression is continuous in $t_1$, so it suffices to show that
\begin{align*}
\int_{U_1 \times U_2^3}
 &
  f(x_1, x_2) f(x_1, x_2+t_2)
f\left(x_1, x_2+c_2t_2+\binom{c_2}{2} u_2\right)
f\left(x_1, x_2+(c_2+1)t_2+\binom{c_2+1}{2}u_2\right) \\
&\quad\ dm_{U_1}(x_1) dm_{U_2}(x_2) dm_{U_2}(t_2) dm_{U_2}(u_2) \\
&\quad \geq \mu(A)^4.
\end{align*}
From H\"older's inequality and the Fubini-Tonelli theorem, one has
$$ \int_{U_1} \left(\int_{U_2} f(x_1,x_2)\ dm_{U_2}(x_2)\right)^4 dm_{U_1}(x_1) \geq \left(\int_{U_1 \times U_2} f\ dm_{U_1 \times U_2}\right)^4 = \mu(A)^4$$
and so it suffices to establish the inequality
\begin{equation}\label{temple}
\begin{aligned}
&\int_{U_2^3} F(x_2) F(x_2+t_2) F\left(x_2+c_2t_2+\binom{c_2}{2} u_2\right) F\left(x_2+(c_2+1)t_2+\binom{c_2+1}{2}u_2\right)\\
& \qquad  dm_{U_2}(x_2) dm_{U_2}(t_2) dm_{U_2}(u_2)\\
&\geq (\int_{U_2} F\ dm_{U_2})^4
\end{aligned}\end{equation}
for any real-valued $F \in L^\infty(U_2)$.  

This inequality can be established by Fourier analysis (cf. \cite{bhk} or \cite{green}), but one can also give a Cauchy-Schwarz-based proof as follows.  The starting point is the identity
$$ (c_2-1) x_2 + (c_2+1) \left(x_2+c_2t_2+\binom{c_2}{2}u_2\right) = (c_2-1) \left(x_2+(c_2+1)t_2+\binom{c_2+1}{2}u_2\right) + (c_2+1) (x_2+t_2).$$
From this and some routine computation we see that for any $x, y, y' \in U_2$, there exists a unique triple $(x_2,t_2,u_2)$ such that
\begin{align*}
(c_2-1) x_2 &= y \\
(c_2+1) (x_2+c_2t_2+\binom{c_2}{2}u_2) &= x-y \\
(c_2-1) x_2+(c_2+1)t_2+\binom{c_2+1}{2}u_2 &= y' \\
(c_2+1) x_2+t_2 &= x-y'
\end{align*}
and so we may rewrite the left-hand side of \eqref{temple} after a linear change of variables as
$$ \int_{U_2} \left(\int_{U_2} F\left( (c_2-1)^{-1} y \right) F\left( (c_2+1)^{-1} (x-y) \right)\ dm_{U_2}(y)\right)^2\ dm_{U_2}(x)$$
(note that $c_2-1,c_2+1$ are invertible in $\F_p$).  By Cauchy-Schwarz, this is greater than or equal to
$$ \left(\int_{U_2} \left(\int_{U_2} F\left( (c_2-1)^{-1} y \right) F\left( (c_2+1)^{-1} (x-y) \right)\ dm_{U_2}(y)\right)\ dm_{U_2}(x)\right)^2$$
which by a further linear change of variables is equal to $(\int_{U_2} F\ dm_{U_2})^4$, giving \eqref{temple}.

\section{Counterexamples}\label{counter-sec}

Now we establish Theorem \ref{counter}.  We begin by passing from sets $A$ to functions $f$, which are more convenient from the perspective of building counterexamples.  More precisely, we use the following ``Bernoulli extension'' construction:

\begin{theorem}[Reduction to the function case]\label{counter-2}  Let $p$ be a prime, let $c_0,\ldots,c_k$ be distinct elements of $\F_p$.  Let $G = \F_p^\omega$, and suppose that there exists an ergodic $G$-system $\mathrm{X} = (X,\X,\mu,(T_g)_{g \in G})$, a non-negative function $f \in L^\infty(\mathrm{X})$, and $\eps>0$ such that the set
$$ \left\{ g \in G: \int_X (T_{c_0 g} f) \ldots (T_{c_k g} f)\ d\mu \geq (\int_X f\ d\mu)^{k+1} - \eps \right\}$$
is not syndetic.  Then $(c_0,\ldots,c_k)$ does not obey the Khintchine property in characteristic $p$.
\end{theorem}

\begin{proof}
We turn to the details.  Let the notation and hypotheses be as in Theorem \ref{counter-2}.   By rescaling, we may assume that $f$ takes values in $[0,1]$.

We now construct a new system $\mathrm{Y} = (Y, \Y, \mu_Y, (S_g)_{g \in G})$ as follows.  The underlying space will be $X \times [0,1]^G$, with the measure $\mu_Y$ given by the product of $\mu_X$ and the $G$-fold product of uniform measure on $[0,1]$, and similarly for the $\sigma$-algebra $\Y$.  The shift $S_g$ will be given by the formula
$$ S_g( (x, (t_h)_{h \in G}) ) := (T_g x, (t_{h+g})_{h \in G} ).$$
One easily verifies that $\mathrm{Y}$ is a $G$-system, and that $(\mathrm{X}, \pi)$ is a factor of this system where $\pi: Y \to X$ is the projection map onto $X$.  Indeed, the system $\mathrm{Y}$ is the product of $\mathrm{X}$ and a Bernoulli system; as the latter system is weakly mixing, any product of such a system with an ergodic system is ergodic (see e.g. \cite[Proposition 4.5]{furst-book} or \cite[Theorem 4.1]{rue}).  In particular, as $\mathrm{X}$ is ergodic, $\mathrm{Y}$ is ergodic also. 

With $f$ being the function in the hypotheses of Theorem \ref{counter-2}, we now define $A \subset Y$ to be the set
$$ A := \{ (x, (t_h)_{h \in G}): t_0 \leq f(x) \}.$$
This is clearly a measurable set in $Y$, and from the Fubini-Tonelli theorem one has $\mu_Y(A) = \int_X f\ d\mu$.  A further application of Fubini-Tonelli gives
$$ \mu_Y( S_{-c_0 g} A \cap \ldots \cap S_{-c_k g} A ) = \int_X (T_{c_0 g} f) \ldots (T_{c_k g} f)\ d\mu $$
for any $g \in G$, and so we obtain the desired counterexample to $(c_0,\ldots,c_k)$ having the Khintchine property.
\end{proof}

To verify the requirements of Theorem \ref{counter-2}, we use the following ``skew shift'' construction, which reduces the task of demonstrating failure of the Khintchine property to the harmonic analysis task of finding a counterexample to a certain integral inequality.

Let $\T := \prod \F_p $ be the compact group formed as the product of countably many copies of $\F_p$.  
Following the notation of Lemma \ref{equidist-2}, we define the set $HP_{c_0,\ldots,c_k}(\T^m)_{\theta} \subset (\T^m)^{k+1}$ for natural numbers $m,k \geq 1$ and elements $c_0,\ldots,c_k$ of $\F_p$ to be the collection of all tuples $(P(c_0),\ldots,P(c_k))$, where $P: \Z \to \T^m$ can be written in components as $P = (P_1,\ldots,P_m)$, and for each $1 \leq i \leq j$, $P_i$ is a polynomial of degree $<i+1$ with $\partial_1^i P_i = \theta_i$.  

\begin{theorem}\label{loo}  Let $p$ be a prime, and let $c_0,\ldots,c_k$ be distinct elements of $\F_p$.  Suppose that there is a natural number $1 \leq m < p$ and a non-negative function $f \in L^\infty(\T^m)$ such that
$$ \int_{HP_{c_0,\ldots,c_k}(\T^m)_\theta} f \otimes\ldots \otimes f \ dm_{HP_{c_0,\ldots,c_k}(\T^m)_{\theta}} < (\int_{\T^m} f\ dm_{\T^m})^{k+1}$$
for all $\theta = (\theta_1,\ldots,\theta_m) \in \T^m$.  Then $(c_0,\ldots,c_k)$ does not have the Khintchine property.
\end{theorem}

\begin{proof}  We will use an explicit Weyl system, analogous to a skew-shift system on a torus.  We identify the group $G$ with the polynomial ring $\F_p[t]$ on one generator $t$, by identifying each generator $e_n$ of $G$ with $t^{n-1}$.  We then embed $\F_p[t]$ in the field $\F_p[t]((\frac{1}{t}))$ of half-infinite Laurent series $\sum_{n=-\infty}^d c_n t^n$, which is a locally compact space using the norm $\|\sum_{n=-\infty}^d c_n t^n\| := p^d$ when $c_d \neq 0$ (and with $\|0\|=0$, of course).  The quotient space $\T := \sum_{n=-\infty}^d c_n t^n/\F_p[t]$ is then a compact abelian group that can be identified with $\F_p^\infty$, and has a Haar probability measure $dm_\T$ (one should view $\F_p[t]$, $\F_p[t]((\frac{1}{t}))$, $\T$ as being characteristic $p$ analogues of $\Z$, $\R$, $\R/\Z$ respectively). 

Let $\alpha$ be an element of $\F_p[t]((\frac{1}{t}))$ which is irrational in the sense that it is not of the form $f/g$ for any $f,g \in \F_p[t]$ with $g$ non-zero.  A simple cardinality (or category, or measure) argument shows that irrational $\alpha$ exist, and it is not hard to give concrete examples of irrational elements.  We then construct a $G$-system $\mathrm{X} = (X,\X,\mu,(T_g)_{g \in G})$  by setting $X$ to be the ``torus'' $X := \T^m$ with the product measure $dm_{\T^m}$ and product Borel $\sigma$-algebra, and shifts
\begin{equation}\label{shifty}
 T_g( (x_i)_{i=1}^{m} ) := ( \sum_{j=0}^i \binom{g}{i-j} x_j )_{i=1}^{m} 
\end{equation}
for all $x_1,\ldots,x_{m} \in \T$, where we adopt the convention that $x_0 := \alpha$, and $\binom{g}{i} := \frac{g(g-1)\ldots(g-i+1)}{i!}$ is viewed as an element of $\F_p[t]$ (note that this is well-defined for any $i<p$, which is acceptable for us since $m<p$).  This shift system is a dynamical abstraction of the binomial identity
$$ \binom{h+g}{i} = \sum_{j=0}^i \binom{g}{i-j} \binom{h}{i}$$
for any $h,g \in G$.  

It is easy to see that $\mathrm{X}$ is a $G$-system.  Now we verify that it is ergodic.  By the ergodic theorem, this is equivalent to the assertion that
\begin{equation}\label{laa}
 \lim_{n \to \infty} \E_{g \in \F_p^n} T_g f = \int_X f\ d\mu
\end{equation}
in $L^2$ norm for all $f \in L^2(\mathrm{X})$, and some F{\o}lner sequence $(\Phi_n)_{n=1}^\infty$ of $G$.  By Fourier decomposition and a density argument, it suffices to achieve this for functions $f$ of the form
\begin{equation}\label{fii-alt}
 f(x_1,\ldots,x_{m}) = e_p( a_1 x_1 + \ldots + a_{m} x_{m} )
\end{equation}
for $a_1,\ldots,a_{m} \in \F_p[t]$, where the standard character $e_p: \T \to \C$ is defined by the formula
$$ e_p\left( \sum_{n=-\infty}^d c_n t^n \operatorname{ mod } \F_p[t] \right) := e^{2\pi i c_{-1}/p}.$$

If all the $a_1,\ldots,a_{m}$ vanish, then both sides of \eqref{laa} are clearly equal to $1$, so the claim is trivial in this case.  Now suppose that there is $1 \leq i_* \leq m$ such that $a_{i_*}$ is non-vanishing, but $a_i=0$ for all $i_* < i \leq m$.  We induct on $i$.  If $i=1$ then we have
$$ \E_{g \in \F_p^n} T_g f = (\E_{g \in \F_p^n} e_p( a_1 \alpha g )) f.$$
As $\alpha$ is irrational, $a_1 \alpha$ does not lie in $\F_p[t]$, and a direct calculation shows that this expression converges to zero as $n \to \infty$.  If $i>1$, we need to show that the left-hand side of \eqref{laa} converges to zero.  By the van der Corput lemma (see\footnote{One could also use the $k=1$ case of Lemma \ref{vdc-2} here.} e.g. \cite[Lemma 2.9]{blm}) it suffices to show that
$$ \lim_{n \to \infty} \E_{g \in \F_p^n} \int_X T_g (T_h f \overline{f})\ d\mu = 0$$
for all $h \in G \backslash \{0\}$.  But one easily verifies that $T_h f \overline{f}$ takes the form \eqref{fii-alt} for some tuple $(a_1,\ldots,a_m)$ which is not identically zero, but vanishes in the $a_{i_*},\ldots,a_m$ entries, so that the claim follows from the induction hypothesis.

For each $g \in G$ and $f_0,\ldots,f_k \in L^\infty(\mathrm{X})$, we consider the quantities
\begin{equation}\label{idef}
I_{c_0,\ldots,c_k;f_0,\ldots,f_k}(g) := \int_{X} (T_{c_0 g} f_0) \ldots (T_{c_k g} f_k)\ d\mu
\end{equation}
and
$$ J_{c_0,\ldots,c_k;f_0,\ldots,f_k}(g) := \int_{HP_{c_0,\ldots,c_k}(\T^m)_\theta} f_0 \otimes\ldots \otimes f_k \ dm_{HP_{c_0,\ldots,c_k}(\T^m)_{\theta(g)}}
$$
where $\theta_i(g) := \binom{g}{i} \alpha$.  We will shortly establish the limit formula
\begin{equation}\label{limit-form}
 \lim_{g \to \infty} |I_{c_0,\ldots,c_k;f_0,\ldots,f_k}(g) - J_{c_0,\ldots,c_k;f_0,\ldots,f_k}(g)| = 0;
\end{equation}
this is very similar to Lemma \ref{equidist-2}, but the limit is in the classical sense (using the one-point compactification $G \cup \{\infty\}$ of $G$, or equivalently using the Frechet filter on $G$) rather than in the uniform density sense.  Assume this formula for the moment.  If $f$ is the function in Theorem \ref{loo} (which we identify with an element of $L^\infty(\mathrm{X})$), we see from the hypotheses on $f$ and compactness that there is an $\eps>0$ such that
$$ \int_{HP_{c_0,\ldots,c_k}(\T^m)_\theta} f \otimes\ldots \otimes f \ dm_{HP_{c_0,\ldots,c_k}(\T^m)_{\theta}} < (\int_{X} f\ d\mu)^{k+1} - \eps$$
for all $\theta \in \T^m$; in particular,
$$ J_{c_0,\ldots,c_k;f,\ldots,f}(g) < (\int_X f\ d\mu)^{k+1} - \eps$$
for all $g\in G$, and thus by \eqref{limit-form}
$$ I_{c_0,\ldots,c_k;f,\ldots,f}(g) < (\int_X f\ d\mu)^{k+1} - \eps/2$$
for all but finitely many $g$.  Applying Theorem \ref{counter-2}, we conclude that $(c_0,\ldots,c_k)$ does not have the Khintchine property.

It remains to establish the limit formula \eqref{limit-form}.  We can repeat large portions of the proof of Lemma \ref{equidist-2} to do this.  Indeed, by the same Fourier decomposition used to prove Lemma \ref{equidist-2}, we may assume that each $f_i$ takes the form
\begin{equation}\label{fi-form}
 f_i(x_1,\ldots,x_m) := e_p( \sum_{j=1}^m a_{ij} x_j )
\end{equation}
for some coefficients $a_{ij} \in \F_p[t]$.  By using Lagrange interpolation identities exactly as in the proof of Lemma \ref{equidist-2}, one can reduce to the case where $a_{ij}=0$ whenever $j \leq i \leq k$, and then reduce further to the case where there is $1 \leq j_* \leq m$ such that $a_{j_*-1,j_*} \neq 0$ and $a_{ij}=0$ whenever $i \geq j$ or $j > j_*$.  As in the proof of Lemma \ref{equidist-2}, $J_{c_0,\ldots,c_k;f_0,\ldots,f_k}(g)=0$ in this case, so it remains to show that
$$ \lim_{g \to \infty} I_{c_0,\ldots,c_k;f_0,\ldots,f_k}(g) = 0.$$
Using \eqref{idef}, \eqref{fi-form}, and \eqref{shifty}, the expression $I_{c_0,\ldots,c_k;f_0,\ldots,f_k}(g)$ can be expanded as
$$ \int_{\T^m} e_p( \sum_{i=0}^{j_*-1} \sum_{j=1}^{j_*} a_{ij} (\sum_{l=0}^j \binom{c_i g}{j-l} x_l) )\ dm_{\T^m}(x_1,\ldots,x_m),$$
with the convention $x_0=\alpha$.  By Fourier analysis, this expression vanishes unless one has
\begin{equation}\label{lsum} 
\sum_{i=0}^{j_*-1} \sum_{j=l}^{j_*} a_{ij} \binom{c_i g}{j-l} = 0
\end{equation}
for all $l=1,\ldots,j_*$.  Thus, it will suffice to show that for all but finitely many $g$, the identities \eqref{lsum} do not simultaneously hold for all $l=1,\ldots,j_*$.

Suppose for contradiction that there are infinitely many $g \in G$ for which \eqref{lsum} holds for all $l=1,\ldots,j_*$.  As the left-hand side of \eqref{lsum} is a polynomial in $g$, these polynomials must then vanish identically for each $l$.  In particular, extracting the $g^{j_*-l}$ coefficient of the left-hand side, we conclude that
$$
\sum_{i=0}^{j_*-1} a_{ij_*} c_i^{j_*-l} = 0$$
for all $l=1,\ldots,j_*$.  But as the Vandermonde determinant of the $c_0,\ldots,c_{j_*-1}$ is non-vanishing, this implies that $a_{ij_*}=0$ for all $i=0,\ldots,j_*-1$, giving the desired contradiction.  This establishes \eqref{limit-form}, and Theorem \ref{loo} follows.
\end{proof}

Now we can prove Theorem \ref{counter}.  Fix $k \geq 3$ and $p$.  We say that a property holds for \emph{generic} tuples $(c_0,\ldots,c_k) \in \F_p^{k+1}$ if the number of tuples which fail to have the property is at most $C_k p^k$ for some $C_k$ depending only on $k$.  Thus, for instance, a generic tuple $(c_0,\ldots,c_k)$ has all entries $c_0,\ldots,c_k$ distinct.  Our task is to establish that a generic tuple $(c_0,\ldots,c_k)$ does not obey the Khintchine property.    In view of Theorem \ref{loo} (applied with $m=2$), it will suffice to locate, for each generic tuple, a non-negative function $f \in L^2(\T^2)$ with the property that
$$ \int_{HP_{c_0,\ldots,c_k}(\T^2)_\theta} f \otimes \ldots \otimes f \ dm_{HP_{c_0,\ldots,c_k}(\T^2)_{\theta}} < (\int_{\T^2} f\ dm_{\T^2})^{k+1}$$
for all $\theta = (\theta_1,\theta_2) \in \T^2$.  The left-hand side can be expanded as
$$ \int_{\T^3} \prod_{i=0}^k f(x_1+c_i \theta_1,x_2+c_i t_2 + \binom{c_i}{2} \theta_2)\ dm_{\T^3}(x_1,x_2,t_2).$$
To create this counterexample, we will use the perturbative ansatz
$$ f(x_1,x_2) = 1 + \sum_{(a_1,a_2) \in A} \eps_{a_1,a_2} e_p( a_1 x_1 + a_2 x_2 )$$
where $A$ is a set of non-zero elements of $\F_p[t]^2$ of bounded cardinality (in fact, in our example we will have $|A|=8$) and $\eps_{a_1,a_2}$ are small complex coefficients to be chosen later.  In order for $f$ to be real-valued, we will need $A$ to be symmetric ($A=-A$) and the coefficients $\eps_{a_1,a_2}$ have to obey the symmetry condition
\begin{equation}\label{lamington}
\eps_{-a_1,-a_2} = \overline{\eps_{a_1,a_2}}
\end{equation}
for all $(a_1,a_2) \in A$.  If $A$ is fixed and the $\eps_{a_1,a_2}$ are chosen sufficiently small, then $f$ will be a non-negative element of $L^\infty(\T^2)$.  As $A$ is assumed to not contain $(0,0)$, $f$ will have mean $1$.  It thus suffices to choose $A$ and $\eps_{a_1,a_2}$ as above, for which we have
\begin{equation}\label{too}
 \int_{\T^3} \prod_{i=0}^k f(x_1+c_i \theta_1,x_2+c_i t_2 + \binom{c_i}{2} \theta_2)\ dm_{\T^3}(x_1,x_2,t_2) < 1
\end{equation}
for all $\theta_1,\theta_2 \in \T$.    The left-hand side of \eqref{too} can be expanded as
\begin{align*}
&\sum_{(a_{1,i},a_{2,i}) \in A \cup \{(0,0)\} \hbox{ for } i=0,\ldots,k} \left(\prod_{i=0}^k \eps_{a_{1,i},a_{2,i}}\right) \\
&\quad\quad \int_{\T^3} e_p\left( \sum_{i=0}^k a_{1,i} \left(x_i + c_i \theta_1) + a_{2,i} (x_2+c_i t_2 + \binom{c_i}{2} \theta_2\right) \right)\ dm_{\T^3}(x_1,x_2,t_2),
\end{align*}
with the convention that $\eps_{0,0} := 1$.  The term when all the $(a_{1,i},a_{2,i})$ vanish is $1$.  As for the other terms, they vanish unless one has the identities
\begin{equation}\label{flail}
\begin{split}
\sum_{i=0}^k a_{1,i} &= 0 \\
\sum_{i=0}^k a_{2,i} &= 0 \\
\sum_{i=0}^k a_{2,i} c_i &= 0
\end{split}
\end{equation}
in which case that term is equal to the expression
$$
\left(\prod_{i=0}^k \eps_{a_{1,i},a_{2,i}}\right) e_p\left( \sum_{i=0}^k a_{1,i} c_i \theta_1 + a_{2,i} \binom{c_i}{2} \theta_2 \right).$$
In order to establish \eqref{too}, it will thus suffice to select (for a generic choice of $(c_0,\ldots,c_k)$) a finite symmetric set $A \subset \F_p^2 \backslash \{(0,0)\}$ and sufficiently small coefficients  $\eps_{a_1,a_2}$ for $(a_1,a_2) \in A$ obeying \eqref{lamington} and the following property:
\begin{itemize}
\item  There is at least one choice of tuple $(a_{1,i},a_{2,i})_{i=0}^k \in (A \cup \{(0,0)\})^{k+1}$, not all vanishing, obeying \eqref{flail}.
Furthermore, for all such tuples, one has the additional constraints
\begin{equation}\label{sand}
\sum_{i=0}^k a_{1,i} c_i = 0
\end{equation}
and
\begin{equation}\label{sand-2}
\sum_{i=0}^k a_{1,i} \binom{c_i}{2} = 0
\end{equation}
and
\begin{equation}\label{sand-3}
 \Re \prod_{i=0}^k \eps_{a_{1,i},a_{2,i}} < 0
\end{equation}
\end{itemize}
Indeed, from the previous discussion, the left-hand side of \eqref{too} will be equal to $1$ plus the left-hand side \eqref{sand-3} for all tuples $(a_{1,i}, a_{2,i})_{i=0}^k$, not all vanishing, obeying \eqref{flail}.

It remains to locate $A$ and $(\eps_{a_1,a_2})_{(a_1,a_2) \in A}$ with the stated properties.  We do this as follows.  We first locate a non-trivial quadruple $(\alpha_0,\alpha_1,\alpha_2,\alpha_3) \in \F_p^4$ with the vanishing properties 
\begin{align}
\sum_{i=0}^3 \alpha_i &= 0 \label{c0}\\
\sum_{i=0}^3 \alpha_i c_i &= 0\label{c1}\\
\sum_{i=0}^3 \alpha_i c_i^2 &= 0\label{c2}
\end{align}
(note that these sums are well-defined when $k \geq 3$).  Indeed, one can use the Lagrange interpolation formula to take
\begin{equation}\label{lio}
\alpha_i := \prod_{0 \leq j\leq 3: j \neq i}\frac{1}{c_i-c_j}.
\end{equation}
Note that generically, the $c_i$ are all distinct, so that the $\alpha_i$ in \eqref{lio} are well-defined and non-vanishing.  We then set
$$ A := \{ \sigma ( \alpha_i c_i, \alpha_i ): i=0,1,2,3; \sigma \in \{-1,+1\}\}.$$
This set is clearly symmetric.  Because the $c_i$ are all distinct, the $\alpha_i$ are all non-zero, and the characteristic $p$ is not equal to two, we see that $A$ consists of eight distinct elements of $\F_p^2 \backslash \{(0,0)\}$.  

Now we classify all the tuples $(a_{1,i},a_{2,i})_{i=0}^k \in (A \cup \{(0,0)\})^{k+1}$ obeying \eqref{flail}.  Certainly the tuple when $(a_{1,i},a_{2,i})=(0,0)$ for all $i$ does this.  By \eqref{c0}, \eqref{c1}, \eqref{c2}, we see that the tuples given by 
$$ (a_{1,i}, a_{2,i}) = \sigma 1_{i \leq 3} (\alpha_i c_i, \alpha_i)$$
for $\sigma = +1, -1$ also obeys \eqref{flail}, as well as \eqref{sand}, \eqref{sand-2}, and will also obey \eqref{sand-3} if one chooses the weights $\eps_{a_{1,i},a_{2,i}}$ so that
$$ \Re \prod_{i=0}^3 \eps_{\alpha_i c_i, \alpha_i} < 0$$
which is easily accomplished.  To conclude the construction, it suffices to show that for generic $(c_0,\ldots,c_k)$ there are no other tuples obeying \eqref{flail}.

Suppose for contradiction that we have another tuple $(a_{1,i},a_{2,i})_{i=0}^k \in (A \cup \{(0,0)\})^{k+1}$ obeying \eqref{flail}.  We can write
$$ (a_{1,i}, a_{2,i}) = 1_{i \in B} \sigma_i (\alpha_{j_i} c_{j_i}, \alpha_{j_i} )$$
for some non-empty $B \subset \{0,\ldots,k\}$, and with $\sigma_i \in \{-1,+1\}$ and $j_i \in \{0,1,2,3\}$ for all $i \in B$.  We can exclude the cases when $B = \{0,1,2,3\}$ and $\sigma_i = \sigma$ and $j_i = i$ for all $i \in B$ and some $\sigma = \{-1,+1\}$, since those tuples were already considered.  As the number of possibilities for $B$, $\sigma_i$, $j_i$ depend only on $k$, it suffices to show that for a \emph{fixed} choice of $B, \sigma_i, j_i$ not of the above form, the condition \eqref{flail} fails for generic $(c_0,\ldots,c_k)$.

Fix $B, \sigma_i, j_i$ as above.    The conditions \eqref{flail} can then be written as
\begin{align}
 \sum_{i \in B} \sigma_i \alpha_{j_i} c_{j_i} &= 0 \label{con1}\\
 \sum_{i \in B} \sigma_i \alpha_{j_i} &= 0 \label{con2}\\
 \sum_{i \in B} \sigma_i \alpha_{j_i} c_{i} &= 0.\label{con3}
\end{align}
Suppose first that $B$ contains an element $i_*$ that lies outside of $\{0,1,2,3\}$.  Then the expression $\sum_{i \in B} \sigma_i \alpha_{j_i} c_i$ can be written as $\sigma_{i_*} \alpha_{j_{i_*}} c_{i_*} + Q$ where the quantity $Q$ does not depend on $c_{i_*}$.  Since $\alpha_{j_{i_*}}$ is generically non-zero, we conclude (after first choosing all $c_i$ for $i\neq i_*$, and then observing that generically the constraint \eqref{con3} can hold for at most one $c_{i_*}$) we see that \eqref{con3} fails for generic $(c_0,\ldots,c_k)$, and we are done in this case.

Thus we may assume that $B \subset \{0,1,2,3\}$.  We now focus on \eqref{con2}, which asserts that a certain linear combination of $\alpha_0,\alpha_1,\alpha_2,\alpha_3$ (with coefficients in $\{-4,-3,-2,-1,0,1,2,3,4\}$) vanish.  From \eqref{lio} we may write
\begin{equation}\label{alph}
 \alpha_i = \pm \frac{1}{V} \prod_{0 \leq i' < i'' \leq 3: i',  i''\neq i} (c_{i'} - c_{i''})
\end{equation}
where $V := \prod_{0 \leq i' < i'' \leq 3} (c_{i'}-c_{i''})$ is the Vandermonde determinant.  Thus, \eqref{con2} can be recast as the assertion that a certain linear combination of the polynomials $\prod_{0 \leq i' < i'' \leq 3: i',  i''\neq i} (c_{i'} - c_{i''})$ for $i=0,1,2,3$ vanish.  But it is easy to see that these polynomials are linearly independent (indeed, they each contain a monomial term that is not present in any of the other three polynomials) and so by the Schwarz-Zippel lemma, any non-trivial linear combination of these polynomials is non-zero for generic $(c_0,\ldots,c_k)$.  The only remaining case is when all the coefficients of $\alpha_0,\alpha_1,\alpha_2,\alpha_3$ in \eqref{con2} vanish.  There are two ways this can happen: either $j_0=j_1=j_2=j_3=j$ for some $j$, or (up to permutation) one has $j_0=j_1=j$ and $j_2=j_3=j'$ and $\sigma_0,\sigma_2 = +1$, $\sigma_1, \sigma_3 = -1$ for some $j \neq j'$.

In the former case $j_0=j_1=j_2=j_3=j$, one can cancel $\alpha_j$ from \eqref{con3} asserts a non-trivial linear constraint between $c_0,c_1,c_2,c_3$ with coefficients in $\pm 1$, which then fails for generic choices of $(c_0,\ldots,c_k)$.  Thus we may assume that $j_0=j_1=j$ and $j_2=j_3=j'$ and $\sigma_0,\sigma_2 = +1$, $\sigma_1, \sigma_3 = -1$.  We then turn to \eqref{con3}, which becomes
$$ \alpha_j (c_0-c_1) + \alpha_{j'} (c_2-c_3) = 0$$
which by \eqref{alph} is a constraint of the form
$$ (c_0-c_1) \prod_{0 \leq i' < i'' \leq 3: i',i''\neq j} (c_{i'} -c_{i''}) = \pm
(c_2-c_3) \prod_{0 \leq i' < i'' \leq 3: i',i''\neq j'} (c_{i'} -c_{i''}).$$
By unique factorization, the two polynomials on the left and right-hand sides here are distinct, so by the Schwartz-Zippel lemma, this identity fails for generic $(c_0,\ldots,c_k)$, and the claim follows.

\begin{remark}  The above arguments give an explicit description of the tuples $(c_0,\ldots,c_k)$ for which the Khintchine property is still possible.  It is likely that a further analysis of these exceptional cases (possibly involving modification of the set $A$ and the weights $\eps_{a_{1,i},a_{2,i}}$ will then resolve the conjecture stated in the introduction, but this seems to require a rather large amount of combinatorial and algebraic case checking, and will not be pursued here.
\end{remark}

\begin{remark} Similar counterexamples can be constructed for $\Z$-systems; they are weaker than those based on the Behrend construction given in \cite{bhk}, although they have the benefit of applying to a wider class of coefficients $c_0,\ldots,c_k$.  We leave the details to the interested reader.
\end{remark}

\providecommand{\bysame}{\leavevmode\hbox to3em{\hrulefill}\thinspace}

\end{document}